\documentclass[a4paper]{article}
\usepackage{amsmath,amsthm,amssymb,latexsym,epsfig,graphicx,enumerate}

\usepackage{graphicx}               
\usepackage{amssymb}
\usepackage{pdfsync}
\usepackage{color}

\def\R{{\mathbb R}}

\def\S{{\mathbb S}}
\def\Rd{{\mathbb {R}^d}}

\def\diam{{\mathop {\rm diam}}}
\def\dist{{\mathop {\rm dist}}}

\def\newchi{\raise2pt\hbox{$\chi$}}
\def\newdag{\raise4pt\hbox{\dag}}
\def\eps{\varepsilon}

\numberwithin{equation}{section}

\newtheorem*{acknowledgements}{Acknowledgements}
\newtheorem{theorem}{Theorem}[section]
\newtheorem{lemma}[theorem]{Lemma}
\newtheorem{corollary}[theorem]{Corollary}

\title{The regularity of the boundary of a multidimensional aggregation patch}
\author{A.\ Bertozzi,  J.\ Garnett, T.\ Laurent, and J.\ Verdera}
\date{}

\begin{document}
\baselineskip=15pt \overfullrule=0pt

\maketitle
\section{Introduction}\label{Introduction}
 Active scalar problems are a wide class of research topics
 in fluid dynamics for which basic
questions of existence and
 regularity pose challenging analysis problems that are often viewed as simpler `analogues' of the
  famous Clay Math Prize Navier-Stokes problem.
   One subclass of such problems are the famous `vortex patches' - exact $L^\infty$ solutions
   of the incompressible inviscid two-dimensional fluid equations in which the scalar vorticity is the characteristic function of an evolving domain.
The classical theory of general $L^\infty$ weak solutions dates back
to Yudovich \cite{Yudovich63} in the early 1960s.  However, the
challenging problem of the long time regularity of the patch
boundary was not settled until the early 1990s by Chemin
\cite{CheminPatch} using methods from para-differential calculus. A
geometric harmonic analysis proof was developed later by the first
author and Constantin in \cite{BC}.

More recent works have studied the dynamics of active scalars with a
gradient flow structure.  This opens up the possibility of using
variational methods, including tools from optimal transport theory.
The general problem has the structure

\begin{equation}
\partial_t {\rho} + \text{div} (\rho v)=0 , \quad v=-\nabla K*\rho
\label{aggeq}
\end{equation}
with many papers written to understand various subclasses of this
problem
\cite{BenedettoCagliotiPulvirenti97,BB10,BCL,BL,BLR,BW,BCM,BDP,BV,BCMo,BCKSV,BuCM,BDi,CDFLS,cr09,DP,GP,HJD,HP,Keller-Segel-70,L,LR1,LR2,LT,LZ,ME,MCO,TB,TBL,tosc:gran:00}.
This equation arises in many applications including materials
science
 \cite{HP,NPS01,Poupaud02}, cooperative control \cite{GP}, granular flow
 \cite{Carrillo-McCann-Villani03,Carrillo-McCann-Villani06,tosc:gran:00}, biological swarms\cite{ME,MEBS,TB,TBL}, vortex densities in superconductors
\cite{Weinan:1994,AS,AMS,DZ,Masmoudi} and chemotaxis
\cite{BCKSV,Keller-Segel-70,BCM,BDP}. Some of the recent literature
has studied finite time singularities and local vs global
well-posedness for both the inviscid case (\ref{aggeq})
\cite{BB10,BCL,BGL,BL,BLR,BV,BS,CCFG,CCFGG,HJD,HB,L} and the cases with
various kinds of diffusion \cite{BRB,BCM,BS,LR1,LR2, SV} in multiple
dimensions. The well-known Keller-Segel problem typically has a
Newtonian potential and linear diffusion. For the non-diffusive
problem  (\ref{aggeq}), of particular interest is the transition
from smooth solutions to weak and measure solutions with mass
concentration.

This paper concerns (\ref{aggeq}) with the special case of $K=N$ the
Newtonian potential. This equation is exactly orthogonal to the
classical 2-D inviscid incompressible fluid equations in vorticity
form, in the sense that the velocity field is perpendicular to that
given by the  Bio-Savart law.  Because of its gradient flow
structure, the model makes sense in all space dimensions and we
consider this problem in general dimension greater than one.  We are
interested in a special class of solutions called ``aggregation
patches''.  These are particular $L^\infty$ weak solutions in which
the density $\rho$ is a time dependent constant times the
characteristic function of an evolving domain. Such solutions only
exist for special kernels such as the Newtonian potential.  For this
potential a previous paper \cite{BLL} established a sharp $C^\gamma$
and $L^\infty$ regularity theory for solutions of (\ref{aggeq}) for
$K=N$ in general dimension.  The proof generalizes the ideas of
\cite{MB} for $C^\gamma$ solutions of the vorticity equation and
\cite{Yudovich63} for the $L^\infty$ theory. That work also
developed a numerical method for computing aggregation patch
boundaries and showed some interesting examples in both two and
three dimensions.  Of note is that these numerical solutions develop
nontrivial geometric singularities at the blowup time - typically
with mass concentrated along a ``skeleton''-like structure.  The
simplest example with a trivial behavior is the collapsing sphere
(or disk in 2-D) which due to symmetry collapses to a dirac mass at
a single point at the blowup time.  However elliptical initial data
yield solutions that collapse onto a line segment and more complex
initial conditions appear to collapse to a structure with branched
arms.  In an analogous fashion, the spreading solutions (backward
time) were also studied both theoretically and numerically.  In
\cite{BLL} the authors develop a rigorous theory proving $L^1$
convergence of the spreading patch solutions to an exact spherically
symmetric similarity solution. However numerical simulations suggest
that the weak $L^1$ convergence theory can not be made sharper due
to the development of defects in the patch boundary in the approach
to the similarity solution.

These nontrivial dynamics open up the natural question of the
regularity of the patch boundary for these aggregation patch
solutions.  In this paper we establish this result, working in
H\"older spaces as was done for the vortex patch boundary problem in
fluids in \cite{CheminPatch,BC}. One key idea in \cite{BC} was a
geometric lemma using cancellation properties of the gradient of the
Biot-Savart kernel on half-disks and yielding a uniform estimate for
the gradient of the velocity field with constants depending
logarithmically on quantities measuring the smoothness of the
boundary. We extend this logarithmic inequality to the
multidimensional case involving even singular integrals and
cancellation on half-spheres. It is worth mentioning that such kind
of estimates have appeared before, without mention to the
logarithmic dependence on constants, in connection with problems of
classical analysis (see, for instance, \cite{MOV} and \cite{MOV2}). See also \cite{CG}, in which
they appear in connection with the Muskat problem.

 A difficulty we have to confront in this paper that did not appear
in the incompressibile fluids case is finding defining functions of
the patch for non-zero times. The defining function for the initial
patch transported by the flow is not smooth, because the field is
not divergence free. We find a genuine smooth defining function
adapted to our context, which leads to a commutator formula
expressing the gradient of the velocity applied to the gradient of
our defining function as a commutator of matrix valued singular
integrals.  A special H\"{o}lder estimate in terms of the uniform norm
of the gradient of the velocity field is derived.
  In addition, to prove our result we first develop the local existence and continuation theory for the patch boundary problem in general dimensions,
  which  requires estimates of the transport map of the patch boundary in local coordinates.
    The technical apparatus needed at this point is more involved than the two dimensional case where one can
    parametrize by the circle, at least in the simply connected
    case.  As a counterpart, we can work without additional effort on domains with
    holes  and even on open sets made of finitely many pieces with disjoint closures.
More details on the structure of the proof and information on the
organization of the paper are given in the next subsections.

\subsection{The main result} \label{sec1}

Let $d \geq 2$ and let $N(y)$ be the fundamental solution of the
Laplace equation in~$\R^d.$  Thus $N(y) = \frac{1}{2\pi} \log |y|$
in dimension~$d =2$ and
$$
N(y) =  - \frac{1} {(d-2) \omega_{d-1}} \, \, \frac{1}{|y|^{d-2}},
\quad d \ge 3,
$$
where $ \omega_{d-1}$ is the $d-1$-dimensional surface measure of
the unit sphere in $\R^d.$ We consider the aggregation equation
\begin{equation}\label{eq1.1}
\frac{\partial \rho}{\partial t} + \operatorname{div}(\rho v) =0
\end{equation}
\begin{equation}\label{eq1.2}
v = -\nabla N * \rho
\end{equation}
with initial data
\begin{equation}\label{eq1.3}
\rho(x,0) = \newchi_{D_0}
\end{equation}
where $\newchi_{D_0}$ is the indicator function of a bounded domain
$D_0 \subset  \R^d.$  We  now fix $0 < \gamma < 1$ and take $D_0$ to
be a bounded $C^{1+\gamma}$ domain (a domain with smooth boundary of
class $C^{1+\gamma}$; a formal definition will be presented in
section \ref{sec2}).  Then we have the following Theorem.

\begin{theorem}\label{theo1.1}
If $D_0$ is a  $C^{1 + \gamma}$ domain, then the initial value
problem~\eqref{eq1.1}, \eqref{eq1.2} and \eqref{eq1.3} has a
solution given by
\begin{equation}\label{eq1.4}
\rho(x,t) = \frac{1}{1 -t} \newchi_{D_t}(x), \quad x \in \Rd, \quad
0 \le t < 1
\end{equation}
where $D_t$ is a  $C^{1 + \gamma}$ domain for all $0 \leq t < 1$.
\end{theorem}

As the proof shows,  the preceding result also holds when $D_0$ is a
union of finitely many  $C^{1 + \gamma}$ domains with disjoint
closures. The conclusion is that $D_t$ is of the same type  for all
$0 \leq t < 1$. It has been recently proved in~\cite{BLL} that the
equation~\eqref{eq1.1}--\eqref{eq1.2} has a unique solution in the
weak sense for each initial condition $\varrho_0(x)$ in
$L^\infty(\R^d) \cap L^1(\R^d).$ If the initial condition is the
indicator function of a bounded domain~$D_0$, then one has
\eqref{eq1.4}. In this case one speaks of aggregation patches, in
analogy with the vortex patches for the vorticity equation
associated with the planar Euler system (see~\cite[Chapter~8]{MB}).
Thus our theorem solves the boundary regularity problem for
aggregation patches. See~\cite{BC}, \cite{CheminPatch} or
\cite[Chapter~8]{MB} for the analogous result for the vorticity
equation in the plane.

We describe a convenient reformulation of the problem that will be
used throughout the rest of the paper.

Set $s=\log(\frac{1}{1-t}),$ so that $0 \le s < \infty$ if and only
if $0 \le t < 1.$ Define
$$
\tilde{\rho}(x,s) = (1-t)\rho(x,t) \quad \text{and} \quad \tilde
v(x,s) = (1-t) v(x,t).
$$
Then, if the initial condition is \eqref{eq1.3}, \eqref{eq1.1}  is
equivalent to the transport equation
\begin{equation}\label{eq1.6}
\frac{\partial \tilde \rho(x,s)}{\partial s} + \nabla \tilde
\rho(x,s) \cdot \tilde v (x,s)= 0.
\end{equation}
The flows (or particle trajectories) in the time variables $t$ and
$s$ are defined respectively by the ODE
$$
\frac{d X(x,t)}{dt} = v( X(x,t),t), \quad X(x,0)=x
$$
and
$$
\frac{d \tilde X(x,s)}{ds} = \tilde v(\tilde X(x,s),s), \quad \tilde
X(x,0)=x.
$$
They are the same, in the sense that $\tilde X(x,s)= X(x,t).$ Hence
the solution of the transport equation~\eqref{eq1.6} with initial
condition $\tilde \rho (x,0) =
\newchi_{D_0}(x)$ is
$$
\tilde \rho(x,t)= \newchi_{\tilde D_s}(x),\quad \tilde D_s = \tilde
X(D_0,s)= X(D_0,t) = D_t = D_{1-e^{-s}}.
$$
Dropping the tildes to simplify the writing and denoting again by
$t$ the new time $s$, we conclude that the problem
\eqref{eq1.1}--\eqref{eq1.3} is equivalent to the non-linear
transport equation
\begin{equation}\label{eq1.7}
\frac{\partial \rho(x,t)}{\partial t} + \nabla  \rho(x,t) \cdot v
(x,t)= 0
\end{equation}
\begin{equation}\label{eq1.8}
v(x,t)= (-\nabla N * \rho)(x)
\end{equation}
with initial condition
\begin{equation}\label{eq1.9}
\rho(x,0) = \newchi_{D_0}(x).
\end{equation}
Theorem~\ref{theo1.1} can be then reformulated as follows.

\begin{theorem}\label{theo1.2}
If $D_0$ is a  $C^{1 + \gamma}$ domain then the initial value
problem~\eqref{eq1.7}, \eqref{eq1.8} and \eqref{eq1.9} has a
solution given by
\begin{equation*}\label{rot}
\rho(x,t) = \newchi_{D_t}(x), \quad x \in \Rd, \quad t \in \R,
\end{equation*}
where $D_t$ is a  $C^{1 + \gamma}$ domain for all $t \in \R$.
\end{theorem}

The problem \eqref{eq1.7}--\eqref{eq1.9} for $d=2$ is similar to the
vorticity equation for incompressible perfect fluids. The difference
is that the velocity field in the vorticity equation is given by
$\nabla^{\perp}N *\rho,$ which is an orthogonal gradient and,
therefore, is divergence free. Instead the field \eqref{eq1.8} has
divergence $-\rho.$

\subsection{Outline of the paper}
The proof of Theorem \ref{theo1.2} is in two steps. First we look at
the ODE giving the flow

\begin{equation}\label{flux0}
\frac{dX(\alpha,t)}{dt} =v (X(\alpha,t),t),\quad X(\alpha,0) =
\alpha, \quad \alpha \in \R^d, \quad t \in \R,
\end{equation}
where the velocity field is
\begin{equation}\label{velocity0}
v(x,t) = - (\nabla N *\chi_{D_{t}})(x),\quad x \in \mathbb{R}^{d},
\quad D_t = X(D_0,t).
\end{equation}
Following Yudovich \cite{Yudovich63} (see ~\cite[Chapter~8]{MB} for
a modern exposition),  the authors in
  \cite{BLL}  prove that \eqref{flux0}-\eqref{velocity0} has a
unique solution and that for each $t$ the mapping $\alpha
\rightarrow X(\alpha,t)$ is a homeomorphism of $\R^d$ onto itself
satisfying a H\"older condition of order $\beta(t),$ with $\beta(t)$
decreasing exponentially to $0$ as $t$ tends to $\infty.$ This does
not use the smoothness of the boundary of the initial domain $D_0$
and, in fact, it holds for an initial condition in
$L^\infty(\R^d)\cap L^1(\R^d),$  with \eqref{velocity0} modified
appropriately.
 Assuming that $D_t$ is a $C^{1+\gamma}$ domain for $t$ in some
time interval around $t=0,$ one can view equation \eqref{flux0} as
an ODE in the Banach space $C^{1+\gamma}(\partial D_0, \Rd).$ This
ODE can be solved for short times by applying the Picard theorem in
$C^{1+\gamma}(\partial D_0, \Rd)$ and thus one gets a flow of
$C^{1+\gamma}$-diffeomorphisms solving
\eqref{flux0}-\eqref{velocity0}. By uniqueness of the Yudovich flow
one concludes that the restriction of $X(\cdot,t)$ to $\partial D_0$
is of class $C^{1+\gamma}$ on the surface $\partial D_0.$ In other
words, the Yudovich flow is, for short times, of class
$C^{1+\gamma}$ in the directions tangential to $\partial D_0.$ This
is discussed in section \ref{sec2}. Theorem  \ref{theo2.1} provides
the local existence result. One should remark that the statement of
Theorem \ref{theo2.1} includes a precise lower bound for the size of
the time interval on which the solution exists, which will be used
later on when dealing with long time existence. Showing that the
Picard theorem can be applied is equivalent to various estimates,
which are collected in Theorem \ref{theo2.2}. Its proof is presented
in sections \ref{sec4} and \ref{sec5}. One needs bounds for the
action of principal value singular integral operators on H\"{o}lder
classes on smooth surfaces. The preliminary Section~\ref{sec3} gives
two well-known lemmas in the precise forms that we need for the
proof of Theorem~\ref{theo2.2}.

The second step consists in proving that the Yudovich flow is of
class $C^{1+\gamma}$ in the directions tangential to $\partial D_0$
for all times. This requires a priori estimates for the quantities
that determine the size of the local existence interval. The most
relevant are those measuring the smoothness of the boundary of a
$C^{1+\gamma}$ domain. The $C^{1+\gamma}$ smoothness of the boundary
of a domain $D_0$ is encoded in a defining function, that is, a
function $\Phi_0$ on class $C^{1+\gamma}$ in $\R^d$, vanishing
exactly on the boundary $\partial D_0$ and with non-zero gradient at
each point of $\partial D_0.$ By transporting $\Phi_0,$ that is, by
setting $\varphi (x,t) = \Phi_0(X^{-1}(x,t))$ one obtains a function
vanishing exactly on $\partial D_t,$ but with gradient $\nabla
\varphi(x,t)= \nabla \Phi_0 \circ \nabla X^{-1}(x,t)$ which may have
a jump at $\partial D_t,$ just as $\nabla X^{-1}(x,t).$ Thus
$\varphi (x,t)$ may not be a defining function for $D_t$ (and,
indeed, it is not), contrarily to what happens in the case of the
vorticity equation in the plane, for which the velocity field is
divergence free. One of the difficulties that we have to overcome is
finding a correct way of changing $\Phi_0$ by means of the flow and
still getting a genuine defining function $\Phi(x,t)$ for $D_t.$
This is done in section \ref{sec8}. Once this is achieved, we need
to get a priori estimates for the $\gamma$-H\"{o}lder semi-norm
$\|\nabla \Phi(\cdot,t)\|_\gamma$ on $\R^d$ and for the infimum of
$|\nabla \Phi(x,t)|$ on $\partial D_t.$ The subtlest estimate is
that of $\|\nabla \Phi(\cdot,t)\|_\gamma,$ which follows by bringing
into the scene an appropriate commutator between a singular integral
and a pointwise multiplication operator. This estimate is performed
in section \ref{sec7}. Once the a priori estimates on the quantities
determining the size of the local existence interval are available,
the $C^{1+\gamma}$ smoothness of $\partial D_t$ for all $t \in \R$
follows readily.

We close this section by showing that the transported defining
function is already not smooth when $d=2$ and the initial patch is
the unit disc $D_0=\{x \in \R^2 : |x| < 1\}$. The solution of the
aggregation equation \eqref{eq1.1}-\eqref{eq1.3} with initial
condition the characteristic function of the unit disc is
\begin{equation*} \rho(x,t)= \frac{1}{1-t}
\,\chi_{D(0,\sqrt{1-t})}(x), \quad x \in \R^2, \quad t < 1.
\end{equation*}
The field is
\begin{equation*}
v(x,t)= -\frac{1}{2}
\begin{cases} \frac{x}{1-t}, & \quad
|x|<\sqrt{1-t},\\*[7pt] \dfrac{x}{|x|^{2}}, &\quad |x|\ge \sqrt{1-t}
\end{cases}
\end{equation*}
and the inverse flow
\begin{equation*}
X^{-1}(x,t)=
\begin{cases} \frac{x}{\sqrt{1-t}}, & \quad
|x|< \sqrt{1-t} ,\\*[7pt] \sqrt{|x|^2+t} \; \frac{x}{|x|}, &\quad
|x| \ge \sqrt{1-t}.
\end{cases}
\end{equation*}
Take as defining function for $D_0$ the function
$\varphi_0(x)=|x|^2-1$. Transporting $\varphi_0$ by the flow we obtain
\begin{equation*}
\varphi(x,t)=\varphi_0(X^{-1}(x,t)) =
\begin{cases}
\frac{1}{1-t} \,\left(|x|^2-(1-t)\right), & \quad |x|< \sqrt{1-t}
,\\*[7pt] |x|^2-(1-t), & \quad |x| \ge \sqrt{1-t},
\end{cases}
\end{equation*}
whose gradient has a jump at the circle $|x|= \sqrt{1-t}$  except
for  $t=0.$ To correct the jump one may take
$$
\Phi(x,t)= (1-t) \chi_{D_t}(x) \varphi(x,t)+ \chi_{\R^2 \setminus
D_t}(x) \varphi(x,t),
$$
where $D_t =D(0, \sqrt{1-t}).$

\section{A Flow of \boldmath$C^{1 + \gamma}$ Surfaces}\label{sec2}

Given $x \in \Rd$ with fixed  $d \geq 2$ the cylinder with center
$x$ and radius $r$ is
$$
C(x,r)= \{x \in \Rd : | y'-x'| \leq r  \quad\text{and}  \quad
|y_d-x_d| \leq r\},
$$
where we use the standard notation $x'=(x_1,...,x_{d-1}) \in
\R^{d-1}.$
 We say that $D$ is a $C^{1+\gamma}$ domain if for each $x \in \partial D$ there exists $r>0$ such that,
  after possibly a rotation around $x$,
$$
C(x,r) \cap \partial D = \{ y \in C(x,r) : y_d = \varphi (y') \}
$$
where $\varphi$ is a function of class $C^{1+\gamma}$ in a ball
$B(x',r'), \; r'>r.$ In other words, the boundary of $D$ is locally
the graph of a $C^{1+\gamma}$ function and thus a surface of class
$C^{1+\gamma}.$  A standard argument based on a partition of unity
shows that if $D$ is
 a $C^{1+\gamma}$ domain then there exists a function $\Phi \in C^{1+\gamma}(\Rd)$ such that $D=\{x \in \Rd : \Phi(x) < 0\}$,
  $\partial D=\{x \in \Rd : \Phi(x) = 0\}$ and $\nabla \Phi(x) \neq 0$ for $x \in \partial D.$
  Such a function is called a defining function for $D$ of class $C^{1+\gamma}.$
  Conversely, by the Implicit Function Theorem, if $D$ has a defining function of class $C^{1+\gamma}$,
  then $D$ is a $C^{1+\gamma}$ domain. There is a very useful quantity measuring the $C^{1+\gamma}$ character of a domain, namely,
\begin{equation}\label{qu}
q(D) = \frac{\|\nabla\Phi\|_{\gamma}}{|\nabla \Phi|_{\inf}}
\end{equation}
 where
\begin{equation*}\label{gradientinf}
|\nabla\Phi|_{\inf} = \inf \{ |\nabla\Phi(x)| : x \in \partial D\}
\end{equation*}
and, for each set $E$ and each function $f$ defined on $E$, we
denote  by $\|f\|_{\gamma, E}$ (or by $\|f\|_{\gamma}$, if there is
no ambiguity on the domain of the function) the H\"{o}lder
$\gamma$-seminorm
$$
\|f\|_{\gamma, E}= \sup \{\frac{|f(x)-f(y)|}{|x-y|^{\gamma}} : x,y
\in E, \; x\neq y \}.
$$

The ODE providing the particle trajectories is
\eqref{flux0}-\eqref{velocity0}.  As we mentioned before, in
\cite{BLL} one proves that \eqref{flux0} has a unique solution and
that for each $t$ the mapping $\alpha \rightarrow X(\alpha,t)$ is a
homeomorphism of $\R^d$ onto itself satisfying a H\"older condition
with exponent $\beta(t)$ decreasing exponentially to $0$ as $t$
tends to $\infty.$ We call $X(\alpha,t)$ the Yudovich flow
associated with the initial condition $\chi_{D_0}.$

Assume that for $t$ in an open interval containing $0$ the
restriction of $X(\cdot,t)$ to $\partial D_0$ is in
$C^{1+\gamma}(\partial{D_0},\R^d).$ Then $\partial D_t$ is a
$C^{1+\gamma}$ domain and we have
\begin{equation*}\label{velocity2}
\begin{split}
v(x,t) &=   -(N *\nabla \chi_{D_{t}})(x) = \int_{\partial D_{t}} N
(x - y) \vec{n} (y)\,d\sigma_{t}(y)
\end{split}
\end{equation*}
where $\vec{n}$ is the exterior unit normal vector to $\partial D_t$
and $d\sigma_t$ is the surface measure on $\partial D_t.$ If
$x=X(\alpha,t)$ and we make the change of variables $y=X(\beta,t)$
we get
\begin{equation*}\label{velocity3}
\begin{split}
v(X(\alpha,t),t) &= \int_{\partial D_{0}} N(X(\alpha,t) - X
(\beta,t)) \Bigl(DX (\beta,t) (T_{1} (\beta)) \wedge \dotsb \wedge
DX (\beta,t) (T_{d-1}(\beta))\Bigr)\,d\sigma (\beta),
\end{split}
\end{equation*}
where $d\sigma$ is the surface measure on $\partial D_0$, $DX$ is
the differential of $X$ as a differentiable mapping from $\partial
D_0$ into $\Rd$ and $T_1(\beta), ..., T_{d-1}(\beta)$ is an
orthonormal basis of the tangent space to $\partial D_0$ at the
point $\beta \in
\partial D_0.$ The vector
\begin{equation}\label{extprod}
\bigwedge^{d-1}_{j=1} DX (\beta, t) (T_{j}(\beta))
\end{equation}
is orthogonal to $\partial D_t$ at the point $X(\beta,t)$ and a
different choice of the orthonormal basis $T_j(\beta), 1\leq j\leq
d-1,$  has the effect of introducing a $\pm$ sign in front of
\eqref{extprod}. We may choose the $T_j(\beta)$ so that
$\vec{n}(\beta), T_{1}(\beta), \dots, T_{d-1}(\beta)$ 
gives the standard orientation of $\mathbb{R}^d$.
Let $\Omega$ be the set of functions $X \in C^{1+\gamma}(\partial
D_0, \Rd)$ such that there exists a constant $\mu \geq 1$ for which
\begin{equation*}\label{mu}
|X(\alpha)-X(\beta)| \geq \frac{1}{\mu}|\alpha-\beta|,\quad
\alpha,\beta \in \partial D_0.
\end{equation*}
The smallest such $\mu$ is denoted by $\mu(X).$  Then $X$ is
bilipschitz and $\mu(X)$ is the Lipschitz constant of the inverse
mapping. It is clear that $\Omega$ is an open subset of
$C^{1+\gamma}(\partial D_0, \Rd).$ Given $X \in \Omega$ set
\begin{equation}\label{defFX}
F(X)(\alpha)= \int_{\partial D_{0}} N(X(\alpha) - X (\beta))
\bigwedge^{d-1}_{j=1} DX (\beta) (T_{j}(\beta)) \,d\sigma (\beta).
\end{equation}
 Therefore $X(\alpha,t)$ satisfies the ODE
\begin{equation}\label{contour}
\frac{dX (\alpha,t))}{dt} = F (X (\cdot,t)) (\alpha), \quad
X(\alpha,0) = \alpha,
\end{equation}
which is called the ``contour dynamics equation". Our plan is to
solve \eqref{contour} for short times in the open subset $\Omega$ of
the Banach space $C^{1+\gamma}(\partial D_0, \Rd)$. By uniqueness of
the trajectories equation \eqref{flux0}-\eqref{velocity0} (see,
e.g., \cite[Theorem 3.7, p.128]{BCD}), we conclude that the
restriction of the Yudovich flow $X({\cdot},t)$ to $\partial D_0$ is of
class $C^{1 + \gamma}$ for short times. In particular, $\partial D_t
= X(\partial D_0,t)$ is a surface of class $C^{1 + \gamma}$ for
short times. In a second step we prove an a priori estimate which
implies that $\partial D_t$ is of class $C^{1 + \gamma}$ for all
times, thus proving Theorem \ref{theo1.2}.

Our estimates are most conveniently performed in terms of a
particular norm defining the topology of $C^{1+\gamma}(\partial D_0,
\Rd).$  Let us momentarily drop the subindex "0" and work with a
bounded $ C^{1+\gamma}$ domain $D$. Since $\partial D$ is a compact
surface of class $ C^{1+\gamma}$ there exists $r=r(D)
>0$ such that for each $\alpha\in
\partial D$ the set $\partial D \cap C(\alpha,r)$ is the graph $\beta_d = \varphi(\beta')$ of a  $
C^{1+\gamma}$ function $\varphi$ (after a rotation around $\alpha$
if needed). The function
$$
\tilde{X}(\beta')= X(\beta',\varphi(\beta')), \quad \beta' \in
B(\alpha',r) \subset \R^{d-1},
$$
is in $ C^{1+\gamma}(B(\alpha',r))$.  Set
$$
\nu(X,\alpha)=
|X(\alpha)|+\|D\tilde{X}\|_{\infty,B(\alpha',r)}+\|D\tilde{X}\|_{\gamma,
B(\alpha',r)},
$$
where $D$ is the ordinary differential of $\tilde{X}$ and for a set
$E$ and a function $f$ on $E$  we denote by $\|f\|_{\infty,E}$ the
supremum norm of $f$ on $E.$
Finally set
$$
\|X\|_{1+\gamma} = \|X\|_{1+\gamma,\, \partial D} = \sup_{\alpha \in
\partial D} \nu(X,\alpha).
$$
Different choices of the local charts yield different but equivalent
norms in $C^{1+\gamma}(\partial D, \Rd).$

We discuss now two simple facts concerning the norm of
$C^{1+\gamma}(\partial D, \Rd)$ we just defined. The first is an
estimate for the norm of the identity mapping $I.$ In the local
chart centered at $\alpha \in \partial D$ we have $\tilde{I}(\beta')
= (\beta',\varphi(\beta')), \; \beta' \in B(\alpha',r).$ Hence
$D\tilde{I}(\beta')$ is a matrix with $d$ rows and $d-1$ columns.
The matrix formed with the first $d-1$ rows is the identity in
$\R^{d-1}$ and the last row is $(\partial_1 \varphi(\beta'),...,
\partial_{d-1} \varphi(\beta') ).$ Since we can assume that $|\nabla \varphi(\beta')| \le 1, \; \beta' \in
B(\alpha',r),$ by implicit differentiation we get (see the proof of
lemma \ref{lem7.4} below) $\|\nabla \varphi \|_{\gamma,
B(\alpha',r)} \le C_d \,q(D).$  Then
$$
\nu(I,\alpha) \le |\alpha|+ C_d + C_d \,q(D).
$$
 Let $c$ be the center of mass of
$D$  and $\operatorname{diam}(D)$ its diameter. Then
\begin{equation}\label{identity}
\|I\|_{1+\gamma, \partial D} \leq \operatorname{diam}(D)+|c|+ C_d +
C_d \,q(D).
\end{equation}
The preceding inequality will be applied to $\partial D_t$ in
dealing with long time existence. On the one hand, the center of
mass is an invariant of the motion, because the kernel $\nabla N$ is
odd. Without loss of generality we assume from now on that the
center of mass of $D_0$ (and, consequently, of $D_t$) is the origin.
On the other hand, we will obtain a priori estimates for
$\diam(D_t)$ and $q(D_t)$ (see \eqref{expdiam} and \eqref{expqD}).
We conclude that the estimate \eqref{identity} is good for our a
priori estimates.

The second fact we should discuss is how one estimates the Lipschitz
constant of $X \in C^{1+\gamma}(\partial D, \Rd)$ in terms of
$\|X\|_{1+\gamma}.$ Take points $\alpha, \beta \in \partial D.$ If
$|\alpha-\beta| < r=r(D)$ we are in a local chart and then clearly
$|X(\alpha)- X(\beta)| \le \|X\|_{1+\gamma} \,|\alpha-\beta|.$
Otherwise we estimate by the uniform norm and we obtain $|X(\alpha)-
X(\beta)| \le 2 \|X\|_{1+\gamma} \le \frac{2 \|X\|_{1+\gamma}}{r} \,
|\alpha-\beta|.$ Hence
\begin{equation}\label{lipX}
|X(\alpha)- X(\beta)| \le (1+ \frac{2}{r(D)})\,\|X\|_{1+\gamma}\,
|\alpha-\beta|, \quad \alpha, \beta \in \partial D.
\end{equation}

\begin{theorem}\label{theo2.1}
The initial value problem
\begin{equation}\label{eqFX}
\frac{d X(\alpha,t)}{dt}= F(X({\cdot},t))(\alpha), \quad
X(\alpha,0)=\alpha,
\end{equation} has a unique solution $X(\alpha,t) \in C^1((-t_0,t_0), C^{1+\gamma}(\partial D_0, \Rd))$  and
$t_0$ depends only on $d, \,q(D_0), \sigma(\partial D_0)$ and
$\operatorname{diam}(D_0).$
\end{theorem}
This follows from the Picard Theorem for Banach spaces and

\begin{theorem}\label{theo2.2}
If $X \in \Omega$, then
\begin{enumerate}
\item[(a)]
\addtocounter{equation}{1}
\begin{equation*}\label{eq2.9}
F(X) \in C^{1+\gamma}(\partial D_0, \Rd)
\end{equation*}
and
\begin{equation}\label{FX}
\|F(X)\|_{1+\gamma} \le C_0 \,\mu(X)^{3d+2}
(1+\|X\|_{1+\gamma}^{2d+4}),
\end{equation}
where $C_0$ denotes a constant that depends only on $d, q(D_0),
\sigma(D_0)$ and $\operatorname{diam}(D_0).$
\item[(b)]
$X \to F(X)$ is locally Lipschitz on $\Omega$: more precisely,
\begin{equation}\label{DiF}
\|DF(X)\|  \leq C_0\,\mu(X)^{3d+8}\, (1+\|X\|_{1+\gamma}^{3d+7}),
\quad X \in \Omega,
\end{equation}
where $C_0$ denotes a constant that depends only on $d, q(D_0),
\sigma(D_0)$ and $\operatorname{diam}(D_0).$
\end{enumerate}
\end{theorem}
Here by $\|DF(X)\| $  we understand the norm of $DF(X)$ as a linear
mapping from $C^{1+\gamma}(\partial D_0,
 \Rd)$ into itself. As the reader will realize, the precise form
 of the constant in the right hand side of \eqref{FX} is
 crucial.

We show now that Theorem \ref{theo2.2} implies Theorem
\ref{theo2.1}. The only point that
 requires further discussion is the size of the interval
 $(-t_0,t_0)$ on which the solution exists. We need to find a ball
$B(I,\rho) \subset \Omega,$  of center the identity and radius
$\rho,$ so that $F$ is
 Lipschitz in $B(I,\rho)$ and we have an explicit bound for $F$ on $B(I,\rho).$
Lemma \ref{lem7.4} gives for $r_0 =r(D_0)$ the inequality
$r_0^{-\gamma} \le 2 q(D_0).$ Take
$$\rho = \frac{1}{2(1+2^{1+\frac{1}{\gamma}} q(D_0)^{\frac{1}{\gamma}})}$$
so that if $X \in B(I,\rho)$ then, by \eqref{lipX},
\begin{equation*}\label{XinOmega}
\begin{split}
|X(\alpha)-X(\beta)| & \geq
|\alpha-\beta|-|X(\alpha)-\alpha-(X(\beta)-\beta)| \\ & \geq
|\alpha-\beta|\, \left(1- (1+\frac{2}{r(D_0)})
\|X-I\|_{1+\gamma}\right) \\& \geq \frac{|\alpha-\beta|}{2},
\end{split}
\end{equation*}
that is, $\mu(X)\le 2, \; X \in B(I,\rho).$

Clearly $\|X\|_{1+\gamma} \le \|I\|_{1+\gamma}+\rho $ for $ X \in
B(I,\rho).$  By \eqref{DiF} $F$ is Lipschitz in $B(I,\rho)$ and by
\eqref{FX} $\|X\|_{1+\gamma} \le C_0\,2^{3d+2}\, \left(1+
(\|I\|_{1+\gamma}+\rho)^{2d+4}\right)$ for $ X \in B(I,\rho).$
Therefore the solution of \eqref{eqFX} exists in an interval
$(-t_0,t_0)$ with
$$t_0 \ge \frac{\rho}{C_0
\,2^{3d+2}\,\left(1+ (\|I\|_{1+\gamma}+\rho)^{2d+4}\right)},$$ which
is a quantity depending only on on $d, q(D_0), \sigma(D_0)$ and
$\operatorname{diam}(D_0).$

\section{Two Lemmas}\label{sec3}

To prove Theorem~\ref{theo2.2} we need two elementary lemmas on
integral operators acting on H\"older functions. These lemmas are
well known but we prove them  for the sake of the reader and because
we need the precise dependence on various constants.

\begin{lemma}\label{lem3.1}
Let $E$ be a measurable subset of $\R^{d-1}$ and assume that $K : E
\times E \rightarrow \R$ is a measurable function on $E \times E$
which satisfies
\begin{align}\label{eq3.3}
|K(x,y)| &\leq \frac{A} {|x-y|^{d-1-\gamma}},\\
\label{eq3.4} |K(x_1,y) - K(x_2,y)| & \leq |x_1 - x_2| \, \frac{A }
{|x_1-y|^{d -\gamma}}, \quad |x_1 - x_2| \leq |x_1 -y|/2.
\end{align}
Then for a constant $C$, which depends only on $d$  and $\gamma$,
\begin{equation}\label{eq3.11}
\left|\left|\int_{E} K(x,y) f(y)\, dy\right|\right|_{\infty, E} \leq
C A  \operatorname{diam}(E)^{\gamma} ||f||_{\infty,E}
\end{equation}
and
\begin{equation}\label{eq3.12}
\left|\left|\int_{E} K(x,y) f(y)\, dy\right|\right|_{\gamma,E} \leq
C A ||f||_{\infty,E}.
\end{equation}
\end{lemma}

\begin{proof} Let $D = \operatorname{diam}(E)$ be the diameter
of $E.$ We assume that $\operatorname{diam}(E)< \infty,$ otherwise
\eqref{eq3.11} is trivially satisfied. If $x \in E,$ then
\begin{equation*}\label{kinf}
\left| \int_{E} K(x,y) f(y)\, dy \right| \le A ||f||_{\infty,E}
\int_{B(x,D)} \frac{dy}{|x-y|^{d-1-\gamma}} \le C A ||f||_{\infty,E}
\, D^{\gamma},
\end{equation*}
which is \eqref{eq3.11}. For \eqref{eq3.12} take two points $x_1,
x_2 \in E$ and set $\delta = 2 |x_1 -x_2|.$ Then
\begin{equation*}\label{kgamma}
\begin{split}
\left| \int_{E} \left(K(x_1,y)- K(x_2,y)\right) f(y)\, dy \right| &
\le  \int_{E \cap B(x_1, \delta)} \left| K(x_1,y) \right| |f(y)|\,dy
+ \int_{E \cap B(x_1, \delta)} \left| K(x_2,y) \right| |f(y)|\,dy
\\ & + \int_{E\setminus B(x_1,
\delta)} \left| K(x_1,y)-K(x_2,y) \right| |f(y)|\,dy \\& =
I_1+I_2+I_3.
\end{split}
\end{equation*}
The first term is estimated as before :
\begin{equation*}\label{kI1}
I_1 \le  C A ||f||_{\infty,E} \,\delta^{\gamma} = C A
||f||_{\infty,E}\, |x_1-x_2|^{\gamma}.
\end{equation*}
The term $I_2$ can be treated as $I_1$ because  $ E \cap B(x_1,
\delta) \subset B(x_2, \frac{3}{2} \delta).$  For $I_3,$ by
\eqref{eq3.4},
\begin{equation*}\label{kI3}
I_3 \le  C A ||f||_{\infty,E} \,|x_1-x_2|  \int_{\R^{d-1} \setminus
B(x_1,\delta)} \frac{dy}{|x_1-y|^{d-\gamma}} \le C A
||f||_{\infty,E}\, |x_1-x_2|^{\gamma}.
\end{equation*}
\end{proof}

\begin{lemma}\label{lem3.2}
Let $E$ be a measurable subset of $\R^{d-1}$ and let $K : E \times E
\rightarrow \R$ be a measurable function on $E \times E$ which
satisfies
\begin{align}\label{eq3.3bis}
|K(x,y)| &\leq \frac{A} {|x-y|^{d-1}},\\
\label{eq3.4bis} |K(x_1,y) - K(x_2,y)| & \leq |x_1 - x_2|\, \frac{A
} {|x_1-y|^{d}}, \quad |x_1 - x_2| \leq |x_1 -y|/2.
\end{align}
Assume that $b$ is a compactly supported bounded measurable function
on $E$ such that for a constant $B$ one has
\begin{equation}\label{bKb}
\|b\|_{\infty,E}+ \sup_{\epsilon > 0} \left| \int_{\{y \in E :
|x-y|>\epsilon \}} K(x,y) b(y)\,dy \right| \le B, \quad x \in E .
\end{equation}
Let $f$ be a function on $E$ satisfying a H\"older condition of
order $\gamma.$ Then
\begin{equation}\label{Kbuni}
\left\| \int_{E} K(x,y) \left(f(x)-f(y)\right) b(y)\, dy
\right\|_{\infty,E} \leq C A B \operatorname{diam}(E)^{\gamma}
||f||_{\gamma,E}
\end{equation}
and
\begin{equation}\label{Kbgam}
\left\| \int_{E} K(x,y)\left(f(x)-f(y)\right) b(y)\,dy
\right\|_{\gamma,E} \le C A B \|f\|_{\gamma,E},
\end{equation}
for some constant $C$ depending only on $d$ and $\gamma.$
\end{lemma}

\begin{proof}
The inequality \eqref{Kbuni} is proved as in the previous lemma. Let
us deal with \eqref{Kbgam}. Given $x_1, x_2 \in E$ set $\delta = 2
|x_1-x_2|.$ Then
\begin{equation*}\label{Kbd}
\begin{split}
& \; \left| \int_{E} K(x_1,y)\left(f(x_1)-f(y)\right) b(y)
\,dy-\int_{E} K(x_2,y)\left(f(x_2)-f(y)\right) b(y) \,dy \right| \\&
\le \int_{E \cap B(x_1,\delta)} |K(x_1,y)| \left|f(x_1)-f(y)\right|
|b(y)| \,dy +\int_{E \cap B(x_1,\delta)} |K(x_2,y)|
\left|f(x_2)-f(y)\right| |b(y)| \,dy \\& + \left| \int_{E \setminus
B(x_1,\delta)} K(x_1,y)\left(f(x_1)-f(y)\right) b(y) \,dy-\int_{E}
K(x_2,y)\left(f(x_2)-f(y)\right) b(y) \,dy \right| \\& =
I_1+I_2+I_3.
\end{split}
\end{equation*}
The first term can be estimated readily by \eqref{eq3.3bis}:
\begin{equation*}\label{Kb1}
I_1 \le C A\,B \,\|f\|_{\gamma,E} \int_{B(x_1,\delta)}
\frac{dy}{|x-y|^{d-1-\gamma}}\,dy  \le C A\,B\,
\|f\|_{\gamma,E}\,|x_1-x_2|^{\gamma}
\end{equation*}
for some constant $C$ depending only on $d$ and $\gamma.$ For $I_2$
one only needs to observe that $E \cap B(x_1,\delta) \subset E \cap
 B(x_2, \frac{3}{2}\delta).$ For $I_3$ we have, by \eqref{eq3.3bis}, \eqref{eq3.4bis} and
 \eqref{bKb},
\begin{equation*}\label{Kbd1}
\begin{split}
I_3 & \le \left| (f(x_1)-f(x_2)) \int_{\{y \in E : |y-x_1|>\delta
\}} K(x_1,y)  b(y) \,dy \right| \\& +  \int_{\{y \in E :
|y-x_1|>\delta \}} \left|f(x_2)-f(y)\right|
\left|K(x_1,y)-K(x_2,y)\right| |b(y)| \,dy
\\& \le B \|f\|_{\gamma,E}\,|x_1-x_2|^{\gamma} +  C A B
\|f\|_{\gamma,E}\,|x_1-x_2| \int_{\{y \in \R^{d-1} : |x_2-y|>
\frac{\delta}{2} \}} \frac{dy}{|x_2-y|^{d-\gamma}} \\ & \le
 C A  B \|f\|_{\gamma,E}\,|x_1-x_2|^{\gamma}.
\end{split}
\end{equation*}
In the second inequality we applied \eqref{eq3.4bis} and then that
$|y-x_1| \ge \frac{2}{3} |y-x_2|$ and $|y-x_2| \ge \frac{\delta}{2}$
\, for $y \in \R^{d-1} \setminus B(x_1,\delta).$
\end{proof}

\section{Proof of Theorem \ref{theo2.2}, part (a)}\label{sec4}

For convenience of notation we assume $d \geq 3.$  The case $d=2$
has a similar proof and can also be obtained, in the simply
connected case, from the argument in \cite[Chapter 8]{MB} when the
tangential derivative $z_{\alpha}(\alpha',t)$ is replaced by the
normal derivative $i z_{\alpha}(\alpha',t).$

Let $X \in \Omega$ and set $\mu= \mu(X). $ By \eqref{defFX}
\begin{equation}\label{FXG}
F(X)(\alpha)= \int_{\partial D_{0}} N(X(\alpha) - X (\beta))\,
 \vec{G}(\beta)\,d\sigma (\beta),  \quad
\alpha \in \partial D_0,
\end{equation}
 where $\vec{G}: \partial D_0 \rightarrow \Rd$ satisfies
\begin{equation*}\label{ge}
\|\vec{G}\|_{\gamma,\partial D_0} \le C_0\|X\|^{d-1}_{1+\gamma}
\end{equation*}
Since $\partial D_0$ is a compact $C^{1 + \gamma}$ surface there
exists $r_0 > 0$ such that for each $\alpha_0 \in \partial D_0$ the
part of $\partial D_0$ lying in the cylinder $C(\alpha_0, 6r_0)$ is
the graph $\alpha_d = \varphi(\alpha')$ of a function $\varphi \in
C^{1+\gamma}(B(\alpha_0', 6r_0)),$  after possibly a rotation around
$\alpha_0.$  We show in Lemma \ref{lem7.4} below that we can take
\begin{equation}\label{r0}
(6 r_0)^{-\gamma} = 2 q(D_0)
\end{equation}
and that one has
\begin{equation*}\label{gradphi}
 \|\nabla \varphi\|_{\gamma, \,B(\alpha_0, 6 r_0)} \le 2 q(D_0).
\end{equation*}
Let $\psi \in C^\infty_0(B(\alpha_0,3r_0))$ such that $0 \le \psi
\le 1,$ $\psi =1$ on $B(\alpha_0,2r_0))$ and $ |\nabla \psi | \le A
/ r_0,$ $A$ a numerical constant. Set $F(X)(\alpha)=F_1(\alpha)+
F_2(\alpha), \; \alpha \in \partial D_0,$ with
\begin{equation}\label{F1}
F_1(\alpha) = \int_{\partial D_0} N(X(\alpha)-X(\beta))\,
\vec{G}(\beta) \,\psi(\beta)\,d\sigma(\beta), \quad \alpha \in
\partial D_0.
\end{equation}
Then $F_1$ is the local part of the integral in \eqref{FXG} and
$F_2$ the far away part. For $\alpha \in
\partial D_0 \cap C(\alpha_0, 6 r_0)$ set, to simplify the writing,
$a=(a_1,...,a_{d-1})= \alpha',$ so that $(a,\varphi(a)) \in \partial
D_0$ for $ a \in B(a_0, 6r_0).$  Define
$$
\tilde{F}_j(a)= F_j(a,\varphi(a)), \quad a \in B(a_0, 6r_0), \quad
j=1,2.
$$
The function  $\tilde{F}_2(a)$ is of class $C^\infty$ in $B(a_0,2
r_0)$ and it is easily estimated in $C^{1+\gamma}(B(a_0,r_0))$ by
taking gradient twice. The result is
\begin{equation*}\label{F2}
\|\tilde{F}_2\|_{1+\gamma,\,B(a_0,r_0)} \le
C_0\,\mu(X)^{d}\,(1+\|X\|^d_{1+\gamma}).
\end{equation*}
The constant $C_0$ in the preceding inequality contains explicitly
the area of the surface $\sigma(\partial D_0)$ and negative powers
of $r_0,$ which can be estimated in terms of $q(D_0)$ by virtue of
\eqref{r0}.

We turn now our attention to the more challenging term
$\tilde{F}_1(a).$ To simplify notation set
\begin{equation}\label{Z}
Z(a)= X(a,\varphi(a)), \quad a \in B(a_0, 6r_0),
\end{equation}
so that
\begin{equation}\label{Zbi}
\frac{1}{\mu(X)}\,|a-b| \le |Z(a)-Z(b)| \le C_0 \|X\|_{1+\gamma}\,
|a-b|, \quad a \in B(a_0, 6r_0).
\end{equation}
Define
$$
M(a,b)= \frac{1}{|Z(a)-Z(b)|^{d-2}}, \quad a, b \in B(a_0, 6r_0).
$$
 An estimate of the norm of $\tilde{F}_1(a)$ in $C^{1+\gamma}(B(a_0,
3 r_0))$ is equivalent to an estimate in this space of the function
$$
Tf(a)= \int M(a,b) f(b)\,db
$$
where
\begin{equation*}\label{ef}
f(b)= \vec{G}(b,\varphi(b)) \psi(b,\varphi(b))\,(1+|\nabla
\varphi(b)|^2)^{1/2}
\end{equation*}
 is in  $ C^{\gamma}(B(a_0, 3 r_0)),$ has compact support in
$B(a_0, 3 r_0),$ and satisfies
\begin{equation}\label{eef}
\|f\|_{\gamma, \,B(a_0, 3r_0)} \le C_0 \, \|X\|_{1+\gamma}^{d-1}.
\end{equation}
Passing to components we can assume that $f$ takes real values. Our
first task is to compute the distributional derivatives of $Tf.$ In
view of the singularity of the kernel and the dimension of the
space, which is $d-1,$ we expect a singular integral of
Calder\'{o}n-Zygmund type to appear. That this is indeed the case is
shown by the formula
\begin{equation}\label{derT}
\partial_j Tf(a)= \operatorname{p.v.} \int \frac{\partial}{\partial a_j} M(a,b) f(b)
\,db, \quad 1 \le j \le d-1,
\end{equation}
involving a principal value integral. The only difficulty in proving
the above identity is to ascertain that the boundary term appearing
in the integration by parts vanishes.  If $g \in
C_0^\infty(B(a,r_0))$ then
\begin{equation*}\label{intparts}
-\int Tf(a) \,\partial_j g(a)\,da =-\lim_{\varepsilon \downarrow 0}
\int \left(\int_{|a-b|>\varepsilon} M(a,b)\, \partial_j
g(a)\,da\right)f(b)\,db.
\end{equation*}
Fix $b$ and integrate by parts to get
\begin{equation}\label{intparts2}
-\int_{|a-b|>\varepsilon} M(a,b) \,\partial_j g(a)\,da =
\int_{|a-b|>\varepsilon}\frac{\partial}{\partial a_j} M(a,b)\,
g(a)\,da +\int_{|a-b|=\varepsilon} g(a)\,M(a,b)\,n_j(a)\,d\sigma(a),
\end{equation}
where $n_j(a)= (a_j-b_j)/|a-b|.$  To handle the boundary term
in~\eqref{intparts2},  note first that since $M(a,b)
=O(|a-b|^{2-d})$ and $\sigma(\partial B(b,\epsilon))=
O\bigl({\epsilon}^{d-2}\bigr)$, we have
$$
\lim_{\eps \downarrow 0}\int_{|a-b| = \eps} g(a) M(a,b) n_j(a)
\,d\sigma(a) = g(b) \lim_{\eps \downarrow 0}\int_{|a-b| = \epsilon}
 M(a,b) n_j(a) \,d\sigma(a).
$$
We now exploit the fact that $n_j$ is an odd function of $\xi= a-b$
to get
\begin{equation*}\label{boundaryterm}
\begin{split}
\int_{|b-a| =\varepsilon} M(a,b)
n_j(a)\,d\sigma(a)&=\int_{|\xi|=\varepsilon} M(b+\xi,b)
\frac{\xi_j}{|\xi|}\,d\sigma(\xi)\\*[7pt] & =\frac{1}{2}
\int_{|\xi|=\varepsilon}
\Bigl(M(b+\xi,b)-M(b-\xi,b)\Bigr)\frac{\xi_j}{|\xi|}\,d\sigma (\xi)
\end{split}
\end{equation*}
 An elementary calculation gives
\begin{equation*}\label{bterm}
\Bigl|M(b+\xi,b)-M(b-\xi,b)\Bigr|\le \frac{C}{|\xi|^{d-1}}
|Z(b+\xi)+Z(b-\xi)-2Z(b)|\le \frac{C}{|\xi|^{d-2-\gamma}}
\end{equation*}
which yields
\begin{equation*}\label{bterm2}
\int_{|a-b| = \eps}
 M(a,b) n_j(a) \,d\sigma(a) = O\bigl({\epsilon}^{\gamma}\bigr),
\end{equation*}
 and, consequently, shows \eqref{derT}.

We now prove that the principal value operator in \eqref{derT} maps
boundedly $ C^{\gamma}(B(a_0, 3 r_0))$ into itself. The strategy is
as follows : if there were the derivative with respect to $b_j$ in
the kernel of the operator in \eqref{derT} we could try an
integration by parts. We show that a sort of commutator changing the
derivative with respect to $a_j$ into one with respect to $b_j$ has
a kernel with extra smoothness and thus the corresponding commutator
operator satisfies the $C^\gamma$ estimate we are looking for.
Set 
\begin{equation*}\label{commu1}
C(a,b)= (2-d) \frac{(Z(a)-Z(b))\cdot
(\frac{\partial}{\partial a_j}Z(a)-\frac{\partial}{\partial b_j}Z(b))} {|Z(a)-Z(b)|^d}.
\end{equation*}

We then have
\begin{equation*}\label{commu}
\begin{split}
\frac{\partial}{\partial a_j} M(a,b) &=(2-d) \frac{(Z(a)-Z(b))\cdot
\frac{\partial}{\partial a_j}Z(a)} {|Z(a)-Z(b)|^d}\\*[7pt]
&=C(a,b)+(2-d) \frac{(Z(a)-Z(b))\cdot \frac{\partial}{\partial
b_j}Z(b)} {|Z(a)-Z(b)|^d}\\*[7pt] &=C(a,b) -\frac{\partial}{\partial
b_j} M(a,b)
\end{split}
\end{equation*}
To show
that the operator
\begin{equation}\label{opcommu}
Cf(a) = \int C(a,b)\,f(b)\,db
\end{equation}
is bounded on $C^\gamma(B(a_0,3r_0))$ we appeal to lemma
\ref{lem3.2}. Remark that
\begin{equation*}\label{kercom}
C(a,b)=K(a,b)\cdot (\frac{\partial Z}{\partial
a_j}(a)-\frac{\partial Z}{\partial b_j}(b))
\end{equation*} with $K(a,b)$ a kernel which satisfies the hypothesis
\eqref{eq3.3} and \eqref{eq3.4} of lemma \ref{lem3.2} with constant
$A=C_0 \|X\|^3_{1+\gamma}\, \mu(X)^{d+2}.$ To apply lemma
\ref{lem3.2} we need to check that
$$
 \left| \int_{|b-a| >\epsilon} K(a,b)\,f(b)\,db \right| \le C, \quad
 \epsilon >0, \quad a \in B(a_0, 3r_0).
$$
 For that write, for $a \in B(a_0,3 r_0),$
\begin{equation*}\label{K1}
\begin{split}
\int_{\varepsilon <|b-a|} K(a,b) f(b)\,db &=\int_{\varepsilon
<|b-a|<6r_0} K(a,b) (f(b)-f(a))\,db +f(a) \int_{\varepsilon <
|b-a|<6r_0} K(a,b)\,db\\*[7pt] &\equiv I(a)+II(a)
\end{split}
\end{equation*}
The term $I(a)$ is estimated straightforwardly by
$$
|I(a)| \le C_0\,\|X\|_{1+\gamma}^{d}\, \mu(X)^{d} \int_{|a-b| < 6
r_0} \frac{1}{|a-b|^{d-1-\gamma}}\,db =  C_0\,\|X\|_{1+\gamma}^d\,
\mu(X)^{d} .
$$
For $II(a)$ we use a ``pseudo-oddness" property of the kernel. We
have
\begin{equation*}\label{psodd}
\begin{split}
&\qquad\int_{\varepsilon<|a-b|<6r_0} \frac{Z(b)-Z(a)}{|Z(b)-Z(a)|^d}
\,db\\*[7pt] &= \int_{\varepsilon<| \xi|<6r_0}
\frac{Z(a+\xi)-Z(a)}{|Z(a+\xi)-Z(a)|^d} \,d\xi\\*[7pt] &=\frac{1}{2}
\int_{\varepsilon<|\xi|<6r_0}\left(
\frac{Z(a+\xi)-Z(a)}{|Z(a+\xi)-Z(a)|^d} + \frac{Z(a-\xi)-Z(a)}
{|Z(a-\xi)-Z(a)|^d}\right)\,d\xi\\*[7pt] &=\frac{1}{2}
\int_{\varepsilon<|\xi|<6r_0} \bigl(Z(a+\xi)-Z(a)\bigr)\left(
\frac{1}{|Z(a+\xi)-Z(a)|^d}
-\frac{1}{|Z(a-\xi)-Z(a)|^d}\right)\,d\xi\\*[7pt] &\quad +
\int_{\varepsilon<|\xi|<6r_0} \frac{Z(a+\xi)+Z(a-\xi)-2Z(a)}
{|Z(a-\xi)-Z(a)|^d}\,d\xi
\end{split}
\end{equation*}
The elementary inequality
\begin{equation*}\label{ztod}
\Big||z|^d-|w|^d \Big| \le d\,\sup_{0 \le j \le d-1}(|z|^{d-1-j},
|w|^j)\,|z \pm w|
\end{equation*}
 provides an estimate for the first term above and the second is estimated straightforwardly. We finally obtain
\begin{equation*}\label{IIa}
\begin{split}
|II(a)| & \le C_0\,\mu(X)^{2d}\,(1+\|X\|^{d+1}_{1+\gamma})\,
\|f\|_{\infty,B(a_0,3r_0)}\\& \le C_0\, \mu(X)^{2d} \, (1+\|X\|^{2
d}_{1+\gamma}).
\end{split}
\end{equation*}
The constant of the kernel $K(a,b),$ as in \eqref{eq3.3} and
\eqref{eq3.4}, is less than
$C_0\,\mu(X)^{d+2}\,\|X\|_{1+\gamma}^{3}.$ Therefore lemma
\ref{lem3.2} yields
\begin{equation}\label{Cop}
\|Cf\|_{\gamma, B(a_0,3 r_0)} \le  C_0\,\mu(X)^{3d+2} \,(1+\|X\|^{2
d+4}_{1+\gamma}).
\end{equation}
It remains to estimate  the operator
\begin{equation*}\label{dbj}
Uf(a)= \operatorname{p.v.} \int \frac{\partial}{\partial b_j} M(a,b)
f(b) \,db
\end{equation*}
on $ C^{\gamma}(B(a_0, 3 r_0)),$ for $1 \le j \le d-1.$  Take a
function $\chi \in C^\infty_0(B(a_0,4r_0))$ such that $0 \le \chi
\le 1,$ $\chi =1$ on $B(a_0,3 r_0)$ and $|\nabla \chi | \le C_0 /
r_0.$ Then $f = f \chi.$ We have
\begin{equation*}\label{U1U2}
\begin{split}
Uf(a)&=\int \frac{\partial}{\partial b_j} M(a,b)
(f(b)-f(a))\chi(b)\,db\\*[7pt] &\quad+f(a) \text{ p.v.} \int
\frac{\partial}{\partial b_j} M(a,b)\chi(b)\,db\\*[7pt]
&=U_1f(a)+f(a)U_2(a).
\end{split}
\end{equation*}
Integrating by parts and noticing that, as before,  the boundary
term vanishes we get
\begin{equation*}\label{U2}
U_2(a)= -\int  M(a,b)\, \frac{\partial}{\partial b_j} \chi(b) \,db,
\quad a \in B(a_0,3 r_0).
\end{equation*}
Thus, by \eqref{Z} and  $\|\nabla \chi\|_{\infty,\,B(a_0, 3r_0)} \le
\,C_0 r_0^{-1},$
\begin{equation}\label{U22}
|U_2(a)| \le C_0 \, \mu(X)^{d-2}, \quad a \in B(a_0,3 r_0),
\end{equation}
and by lemma \ref{lem3.1}
\begin{equation*}\label{U23}
\|U_2\|_{\gamma,\,B(a_0, 3r_0)} \le C_0 \,
\mu(X)^{d-1}\,\|X\|_{1+\gamma}^2.
\end{equation*}
Here the constant of the kernel $M(a,b)$ has been estimated by
$C_0\,\mu(X)^{d-1}\,\|X\|_{1+\gamma}^2.$ By \eqref{eef}
\begin{equation}\label{fU2}
\|f U_2\|_{\gamma,\,B(a_0, 3r_0)} \le C_0 \,
\mu(X)^{d-1}\,\|X\|_{1+\gamma}^{d+1}.
\end{equation}

For $U_1 f$ we apply lemma \ref{lem3.2}. The kernel of the operator
$U_1$ is $\partial /\partial b_j M(a,b)$, whose constant turns out
to be not greater than
$C_0\,\mu(X)^{d+2}\,(1+\|X\|_{1+\gamma}^{4}).$   Tacking into
account \eqref{eef} and \eqref{U22} lemma \ref{lem3.2} yields
\begin{equation}\label{U1}
\|U_1 f \|_{\gamma,\,B(a_0, 3r_0)} \le C_0 \, \mu(X)^{2 d}\,(1+
\|X\|_{1+\gamma}^{d+3}).
\end{equation}
Combining \eqref{fU2} and \eqref{U1}
\begin{equation}\label{eU}
\|U f \|_{\gamma,\,B(a_0, 3r_0)} \le C_0 \, \mu(X)^{2 d}\,(1+
\|X\|_{1+\gamma}^{d+3}).
\end{equation}
By \eqref{eU} and \eqref{Cop} we finally obtain
\begin{equation*}\label{djTf}
\|\partial_j Tf(a)\|_{\gamma,\,B(a_0, 3r_0)} \le  C_0\,\mu(X)^{3d+2}
\,(1+\|X\|^{2 d+4}_{1+\gamma}),
\end{equation*}
which completes the proof of \eqref{FX}.

\section{Proof of Theorem \ref{theo2.2}, part (b)}\label{sec5}

It is enough to show that for $X \in \Omega$ and $H \in
C^{1+\gamma}(\partial D_0, \Rd)$
$$
DF(X)(H)  = \frac{d} {d\lambda} F(X + \lambda H)\Bigr|_{\lambda =0}
$$
satisfies
\begin{equation}\label{DFXe}
||DF(X)(H)||_{1+\gamma}  \leq     C_0\,\mu(X)^{3d+8} \,(1+\|X\|^{3
d+7}_{1+\gamma})   \,||H||_{1+\gamma}.
\end{equation}
To prove this we may assume that $||H||_{1+\gamma} =  1.$   We first
compute
\begin{equation*}\label{DF}
\begin{split}
DF(X)(H)&=\frac{d}{d\lambda} F(X+\lambda
H)\Bigr|_{\lambda=0}\\*[7pt] &=\frac{d}{d\lambda}\Bigr|_{\lambda=0}
\int_{\partial D_0} N\Bigl((X+\lambda H)(\alpha)-(X+\lambda
H)(\beta)\Bigr) \bigwedge^{d-1}_{j=1} D(X+\lambda H)(T_j
(\beta))\,d\sigma(\beta)\\*[7pt] &=\int_{\partial D_0}
N(X(\alpha)-X(\beta))\sum^{d-1}_{j=1} (-1)^{j-1} DH(T_j(\beta))
\bigwedge_{k\ne j} DX(T_k(\beta))\,d\sigma(\beta)\\*[7pt] &\quad
+\int_{\partial D_0} \nabla N (X(\alpha)-X(\beta))\cdot
(H(\alpha)-H(\beta))\bigwedge_{j=1}^{d-1} D(X)
(T_j(\beta))\,d\sigma(\beta)\\*[7pt] &\equiv A(\alpha)+B(\alpha)
\end{split}
\end{equation*}
Consider first the term $A(\alpha).$ This is a sum of $d-1$ terms,
each of which looks like the function $F(X)(\alpha)$ in
\eqref{defFX}. The only difference is that in $A(\alpha)$ one of the
factors $DX(T_j)(\beta)$ has been replaced by a vector of the type
$DH(T_j)(\beta).$ Then the estimate of $A(\alpha)$ in
$C^{1+\gamma}(\partial D_0, \Rd)$ is performed in exactly the same
way as we did in the previous section for $F(X).$   There is only
one difference, namely, that in the bounding terms one of the
factors  $\|X\|_{1+\gamma}$ should be replaced by
$\|H\|_{1+\gamma}=1.$ Thus
$$
\|A\|_{1+\gamma} \le C_0\,\mu(X)^{3d+2}
\,(1+\|X\|^{2d+3}_{1+\gamma}) .
$$

The term $B(\alpha)$ is slightly different because of the presence
of the factor $H(\alpha)-H(\beta)$ in the kernel, which compensates
the higher singularity of $\nabla N(X(\alpha)-X(\beta)).$ The
structure of the argument is, however, the same. One performs the
splitting into local and far away parts, as in \eqref{F1}. The local
part, which is the most difficult, can be written in local
coordinates $a=(a_1,...,a_{d-1})$ as
\begin{equation*}\label{local}
Tf(a)= \int M(a,b) \,f(b)\,db
\end{equation*}
where $f$ is a scalar function satisfying the estimate \eqref{eef},
and the kernel $M(a,b)$ is given by
\begin{equation*}\label{localkernel}
M(a,b)= \nabla N(Z(a)-Z(b)){\cdot}(h(a)-h(b)),
\end{equation*}
$Z(a)= X(a,\varphi(a))$ and $h(a)=H(a,\varphi(a)).$  The function
$Z$   satisfies \eqref{Zbi} and $\|h\|_{1+\gamma, B(a_0,3r_0)} \le
C_0.$ As before, the boundary term vanishes and we have
\begin{equation*}\label{derT2}
\partial_j Tf(a)= \operatorname{p.v.} \int \frac{\partial}{\partial a_j} M(a,b) f(b)
\,db, \quad 1 \le j \le d-1.
\end{equation*}
We express this operator as a commutator minus an operator with
kernel $\partial_j /\partial b_j M(a,b).$  Recall that the
commutator gives the worst constants. The kernel of the commutator
is
\begin{equation*}\label{commu2}
\begin{split}
C(a,b) & = \nabla^2 N \big(Z(a)-Z(b)\big) \big(\partial_j Z(a)-
\partial_j Z(b)\big) {\cdot} (h(a)-h(b)) + \nabla N(Z(a)-Z(b)) {\cdot} (\partial_j
h(a)-\partial_j h(b))\\& = K(a,b)(\partial_j Z(a)- \partial_j Z(b))
 + \nabla N(Z(a)-Z(b)) {\cdot} (\partial_j h(a)- \partial_j h(b)),
\end{split}
\end{equation*}
where the second identity defines the matrix $K(a,b).$ The operator
given by the kernel in the second term
$$
\nabla N(Z(a)-Z(b)) {\cdot} (\partial_j h(a)- \partial_j h(b))
$$
is estimated as we did in the previous section for \eqref{opcommu}.
The worst constants appear in estimating the operator with the
kernel
$$
K(a,b)(\partial_j Z(a)- \partial_j Z(b)).
$$
We follow closely the argument for the estimate of \eqref{opcommu}.
The step that gives the largest constant happens when dealing with
the quantity
\begin{equation}\label{pseudo2}
f(a)\,\int_{\epsilon <| b-a|< 6r_0} K(a,b)\,db, \quad  a \in B(a_0,
3r_0).
\end{equation}
 The pseudo-oddness property of $K$ gives
$$
|\int_{\epsilon <| b-a|< 6r_0} K(a,b)\,db| \le C_0
\,\|h\|_{1+\gamma}\, \mu(X)^{2d+4}\, (1+\|X\|_{1+\gamma}^{d+4}),
\quad a \in B(a_0, 3r_0),
$$
which, combined with \eqref{eef}, yields the upper bound
$$
 C_0
\, \mu(X)^{2d+4}\, (1+\|X\|_{1+\gamma}^{2d+3})
$$
for the norm of the matrix in \eqref{pseudo2}.
 The constant of the
kernel $K(a,b)$ in applying lemma \ref{lem3.2} is $C_0 \,
\mu(X)^{d+4}\,\|X\|_{1+\gamma}^4.$  Thus lemma \ref{lem3.2} finally
gives
$$
\|Tf\|_{1+\gamma} \le C_0\,
\mu(X)^{3d+8}\,(1+\|X\|_{1+\gamma}^{2d+7})
$$
and
\begin{equation*}\label{DFXe2}
||DF(X)(H)||_{1+\gamma}  \leq     C_0\,\mu(X)^{3d+8} \,(1+\|X\|^{3
d+7}_{1+\gamma}),
\end{equation*}
which is \eqref{DFXe}, because we are assuming that
$||H||_{1+\gamma} = 1.$

\section{The logarithmic inequality for  \boldmath$\|\nabla v(\cdot,t)\|_\infty$}  \label{sec6}

Fix a bounded $C^{1 + \gamma}$ domain $D$ and write
\begin{equation}\label{eq7.1}
v(x) = -\nabla N * \newchi_{D}(x).
\end{equation}
In section we will establish a logarithmic estimate for $||\nabla v
||_{\infty}$ that will be needed to get long time solutions of the
problem \eqref{eq1.7}, \eqref{eq1.8} and \eqref{eq1.9}.

\begin{lemma}\label{lem7.1}
Let $x \notin  \partial D$, and let $\epsilon = \epsilon(x) =
\dist(x,\partial D).$ Then for $1 \leq j,k \leq d$, the vector $v =
(v^1, v^2, \dotsc, v^d)$ satisfies
\begin{equation}\label{eq7.2}
\frac{\partial v^j}  {\partial x_k}(x) =
 \left(\frac{d}{\omega_{d-1}}  \frac{x_jx_k}  {|x|^{d+2}} * \newchi_{D\setminus B(x,\epsilon)}\right)(x),\quad j \neq k
\end{equation}
and
\begin{equation}\label{eq7.3}
\frac{\partial v^j}  {\partial x_j}(x)= - \frac{1} {d}
\newchi_{D}(x) - \left(\frac{1}{\omega_{d-1}}\frac {|x|^2 - d \,x_j^2}  {|x|^{d+2}} *
\newchi_{D \setminus B(x,\epsilon)}\right)(x)
\end{equation}
where $\omega_{d-1} = \sigma(\S^{d-1})$ and the derivatives are
distributional derivatives.
\end{lemma}

\begin{proof}
Suppose $j \neq k$.  By \eqref{eq7.1} and Green's theorem,
$$
\frac{\partial v^j}  {\partial x_k}(x) = \frac{d}{\omega_{d-1}}
\lim_{\eta \to 0} \int_{D\cap \{|y-x| > \eta\}} \frac{(x_j -
y_j)(x_k - y_k)} {|x-y|^{d+2}}\,dy,
$$
but for $0 < \eta < \epsilon(x)$
$$
\int_{\eta < |y-x| < \epsilon} \frac{(x_j - y_j)(x_k - y_k)}
{|x-y|^{d+2}}\, dy =0.
$$
That established \eqref{eq7.2}, and the proof of \eqref{eq7.3} is
similar.
\end{proof}

\bigskip
Notice that the principle value  kernels in  \eqref{eq7.2} and
\eqref{eq7.3}  have the form
\begin{equation}\label{eq7.4}
K(x) = \frac{\Omega(x)}  {|x|^{d}}, \quad x \neq 0,
\end{equation}
where
\begin{enumerate}
\item[(i)] $\Omega$ is homogeneous of degree $0$, $\Omega(x) = \Omega(\frac{x}  {|x|}),$
\item[(ii)] $\Omega$ is even, $\Omega(-x) = \Omega(x),$
\item[(iii)] $\Omega \in C^1(\R^d \setminus \{0\}),$
\end{enumerate}
and
\begin{enumerate}
\item[(iv)]  $\int_{\S^{d-1}} \Omega (x) \,d\sigma (x) =0.$
\end{enumerate}

As we mentioned before, a  $C^{1 + \gamma}$ domain $D$ has a
defining function, that is, a $C^{1 + \gamma}$ function $\Phi \colon
\R^d \rightarrow \R$ such that  and $D = \{\Phi < 0\},$ $\partial D
= \{\Phi = 0\}$ and $\nabla \Phi(x) \neq 0$, $x \in
\partial D. $  We set
$$
|\nabla \Phi|_{\rm {inf}} = \inf_{x \in \partial D}|\nabla \Phi(x)|
$$
and
$$
||\nabla \Phi||_{\gamma} = \sup_{x_1 \neq x_2 \in \R^d}
\frac{|\nabla \Phi(x_1) - \nabla \Phi(x_2)|}  {|x_1 -
x_2|^{\gamma}}.
$$

\begin{theorem}\label{theo7.2}
Let $K$ satisfy \eqref{eq7.4} and (i)--(iv)  and let $D$ be a $C^{1
+ \gamma}$ domain with defining function $\Phi.$ Then
\begin{equation}\label{eq7.5}
\left| \int_{ |y-x|> \epsilon} K(x-y) \newchi_D(y)\, dy \right| \leq
\frac{C_d\, C(\Omega)}  {\gamma} \, \left(1 + \log^{+}
\left(|D|^{1/d} \frac{||\nabla \Phi||_{\gamma}}  {|\nabla \Phi|_{\rm
{inf}}}\right)\right), \quad x \in \R^d, \quad 0 < \epsilon,
\end{equation}
where $C_d$ and $C(\Omega)$ are constants that depend only on $d$
and $\Omega$ respectively, and $\log^+ x = \max\{\log x, 0\}$ is the
positive part of the logarithm.
\end{theorem}

\begin{corollary}\label{coro7.3}
If $v(x)$ is defined by \eqref{eq7.1} and $\Phi$ is a defining
function for the bounded $C^{1 + \gamma}$ domain $D$, then
\begin{equation}\label{eq7.6}
||\nabla v||_{\infty} \leq \frac{C'_d}  {\gamma}  \left(1 + \log^+
\left(|D|^{1/d} \frac{||\nabla \Phi||_{\gamma}}  {|\nabla \Phi|_{\rm
{inf}}} \right)\right).
\end{equation}
\end{corollary}

The Corollary is immediate from \eqref{eq7.5}, \eqref{eq7.2} and
\eqref{eq7.3}.

We need a lemma.

\begin{lemma}\label{lem7.4}
Let $D$ be a $C^{1 + \gamma}$ domain with defining function $\Phi.$
If $\delta > 0$ satisfies
$$
\delta^\gamma \frac{||\nabla \Phi||_{\gamma}} {|\nabla \Phi|_{\rm
{inf}}} \le \frac{1}{2},
$$
then for each $x \in \partial D$, $\partial D \cup B(x,\delta)$ is,
after a rotation around $x,$ the graph of a  $C^{1 + \gamma}$
function and $D \cap B(x,\delta)$ is the part of $B(x,\delta)$ lying
below the graph.
\end{lemma}

\begin{proof}
Assume, without loss of generality, that $x=0$ and that
$\nabla\Phi(0)= (0,\dotsc,0, \partial_d \Phi(0))$, $\partial_d
\Phi(0) > 0.$ Take two points $p, q \in \partial D \cup B(0,\delta)$
and set $p=(x',x_d)$ with $x'=(x_1,\dotsc,x_{d-1})$, and
$q=(y',y_d)$ with $y'=(y_1,\dotsc,y_{d-1}).$ Then
$$
0=\Phi(p)=\Phi(0)+\nabla\Phi(0)\cdot p +E(p)=
|\nabla\Phi(0)|x_d+E(p)
$$
and similarly for $q.$ Subtracting and taking absolute value
$$
|\nabla\Phi(0)||x_d-y_d| = |E(p)-E(q)| \le \|\nabla\Phi\|_\gamma \,
\delta^\gamma\,(|x'-y'|+|x_d-y_d|)
$$
and thus
$$
\frac{|x_d -y_d|}{|x'-y'|} \le \frac{\|\nabla\Phi\|_\gamma}{|\nabla
\Phi|_{\inf}}\, \delta^\gamma \,  \left(1+ \frac{|x_d
-y_d|}{|x'-y'|}\right),
$$
which yields
$$
\frac{|x_d -y_d|}{|x'-y'|} \le 1.
$$
This says that $\partial D \cup B(x,\delta)$ is the graph of a
Lipschitz function $x_d = \varphi(x'),$ with domain an open
subset~$U$ of $\{x' \in \R^{d-1}: |x'| < \delta\},$ satisfying
$|\nabla\varphi(x')| \le 1$, $ x' \in U.$ By the implicit function
theorem $\varphi$ is of class $C^{1 + \gamma}$ on its domain and
this completes the proof of the Lemma.

Notice that $U$ contains the ball $\{x' \in \R^{d-1}: |x'| <
\delta/2^{1/2}\}.$  We need also the estimate
\begin{equation}\label{eq7.7}
|\nabla\varphi(x')| \le (2d)^{1/2} \frac{||\nabla \Phi||_{\gamma}}
{|\nabla \Phi|_{\inf}}\,r^\gamma, \quad |x'| \le r <
\frac{\delta}{{2}^{1/2}}.
\end{equation}
By implicit differentiation
$$
\partial_j \varphi(x')=- \frac{\partial_j \Phi(x',\varphi(x'))}{\partial_d
\Phi(x',\varphi(x'))}, \quad 1 \le j \le d-1,
$$
and so
$$
|\partial_j \Phi(x',\varphi(x'))| \le
|\partial_d\Phi(x',\varphi(x'))|, \quad 1 \le j \le d-1,
$$
which gives
$$
|\nabla\Phi(x',\varphi(x'))| \le d^{1/2}\,|\partial_d
\Phi(x',\varphi(x'))|.
$$
Since $\partial_j \Phi(0)=0$, $1 \le j\le d-1,$
$$
|\nabla\varphi(x')| \le
d^{1/2}\frac{\|\nabla\Phi\|_\gamma\,(2^{1/2}r)^\gamma}{|\nabla
\Phi|_{\inf}},\quad |x'| \le r,
$$
which completes the proof of \eqref{eq7.7}.
\end{proof}

\begin{proof}[Proof of Theorem \ref{theo7.2}]
Assume first that $x \in
\partial D.$ Take $\delta >0$ such that $\delta^\gamma \frac{||\nabla \Phi||_{\gamma}} {|\nabla \Phi|_{\rm
{inf}}} = \frac{1}{2}$ and set $\eta= \delta/2^{1/2}.$  Let
$\epsilon$ satisfy $0 <\epsilon < \eta.$ Then
\begin{multline*}
\int_{|y-x| > \epsilon} K(x -y) \newchi_D(y) \,dy = \int_{D \cap
\{\epsilon < |y-x| < \eta\}} K(x -y)\, dy
+ \int_{{D \cap \{\eta < |x-y| < |D|^{1/d}\}}} K(x -y)\,dy   \\
+\int_{D \cap \{|y-x| > |D|^{1/d}\}} K(x -y) \,dy =
  I_1 + I_2 + I_3.
\end{multline*}
 Thus
\begin{equation*}\label{eq7.8}
|I_3| \leq \int_{D \cap \{|y-x| > |D|^{1/d}\}}
\frac{||\Omega||_{\infty}}  {|x-y|^d}\, dy \leq
\frac{||\Omega||_{\infty}}  {|D|} |D|
 = ||\Omega||_{\infty},
\end{equation*}
and
\begin{equation*}\label{eq7.9}
|I_2| \leq ||\Omega||_{\infty} \omega_{d-1} \int_{\eta}^{|D|^{1/d}}
\frac{dr} {r} = ||\Omega||_{\infty} \omega_{d-1} \log
\left(\frac{|D|^{1/d}} {{\eta}}\right) \le C_d \, C(\Omega) \left(1+
\log^+ \left(|D|^{1/d} \frac{\|\nabla \Phi \|_{\gamma}} {|\nabla
\Phi|_{\inf}} \right) \right).
\end{equation*}
If $\eta \ge |D|^{1/d},$ then we let $I_2 =0$ and this brings in
again the positive part of the logarithm.

Let us turn to $I_1.$ Assume that $x=0$ and that we are in the
situation discussed in the proof of Lemma~\ref{lem7.4}. Taking polar
coordinates we get
$$
|I_1|\le \left|\int_{\epsilon}^\eta  \int_{A(r)} \Omega(\omega)
\,d\sigma(\omega) \,\frac{dr}{r}\right|,
$$
where $A(r)= \{\omega \in S^{d-1} : r \omega \in D \}.$  Let $H$
stand for the half-space $\{x \in \R^d : x_d < 0\}.$ Since $\Omega$
is even and has zero integral on $S^{d-1}$, we conclude that the
integral of $\Omega$ on the hemisphere $S^{d-1} \cap H$ is also
zero. Hence
$$
\int_{A(r)} \Omega(\omega) \,d\sigma(\omega) = \int_{B(r)}
\Omega(\omega) \,d \sigma(\omega)-\int_{C(r)} \Omega(\omega)
\,d\sigma(\omega),
$$
where $B(r)= \{\omega \in S^{d-1} : r \omega \in D \setminus H\}$
and  $C(r)= \{\omega \in S^{d-1} : r \omega \in H\setminus D\}.$ Let
us proceed to estimate the integral on $B(r)$ (the integral on
$C(r)$ is estimated similarly).

For some absolute constant $C_0$ (which can be taken to be $\pi/2$)
one has, by \eqref{eq7.7},
$$
\sigma(B(r)) \le C_0 \sup_{|x'|\le r} |\varphi(x')| \frac{1}{r} \le
C_0 \sup_{|x'|\le r} |\nabla \varphi(x')| \le C_0 (2d)^{1/2}
\frac{||\nabla \Phi||_{\gamma}}  {|\nabla \Phi|_{\inf}} \, r^\gamma.
$$
 Therefore
$$
\left|\int_{\epsilon}^\eta  \int_{B(r)} \Omega(\omega)
\,d\sigma(\omega) \frac{dr}{r}\right| \le ||\Omega||_{\infty}
\int_{\epsilon}^\eta \sigma(B(r)) \frac{dr}{r} \le
||\Omega||_{\infty} \,C_0 (2d)^{1/2} \,\frac{||\nabla
\Phi||_{\gamma}}  {|\nabla \Phi|_{\inf}}\frac{1}{\gamma}\,
\eta^\gamma \le C_d\,C(\Omega) \frac{1}{\gamma},
$$
which completes the proof for $0< \epsilon < \eta.$ If $\eta <
\epsilon $, then $I_1 =0$ and  $I_2$ and $I_3$ are estimated as
before.

Let us assume now that $x \notin \partial D.$ Let $\epsilon_0$
denote the distance from $x$ to $\partial D.$ If $\epsilon <
\epsilon_0,$ then
$$
\int_{|y-x| > \epsilon} K(x -y) \newchi_D(y)\, dy = \int_{|y-x| >
\epsilon_0} K(x -y) \newchi_D(y) \,dy
$$ and so we can assume that $\epsilon_0 \le \epsilon.$ Take $x_0
\in \partial D$ with $|x-x_0|=\epsilon_0$ and define
$$
\Delta = \int_{|y-x| > \epsilon} K(x -y) \newchi_D(y)\, dy
-\int_{|y-x_0| > 2 \epsilon} K(x_0 -y) \newchi_D(y)\, dy.
$$
Then $\Delta=\Delta_1 + \Delta_2$, where
$$
\Delta_1 = \int_{|y-x_0| > 2 \epsilon} \Bigl(K(x -y)-K(x_0-y) \Bigr)
\newchi_D(y) \,dy
$$
and
$$
\Delta_2 = \int_{B(x_0,2\epsilon)\setminus B(x,\epsilon)} K(x -y)
\newchi_D(y)\,dy.
$$
We then have
$$
|\Delta_2| \le ||\Omega||_{\infty}
\frac{|B(x_0,2\epsilon)|}{\epsilon^d} = C_d \,||\Omega||_{\infty}
$$
and, by a gradient estimate,
$$
|\Delta_1| \le C(\Omega)\, |x-x_0|\int_{2\epsilon <
|y-x_0|}\frac{dy}{|x-x_0|^{d+1}} \le C_d\,C(\Omega),
$$
which completes the proof of Theorem \ref{theo7.2}.
\end{proof}

\section{Global Existence}\label{sec7}

We prove in this section that the Yudovich flow $X(\alpha,t)$
solving the ODE \eqref{flux0} and \eqref{velocity0} is smooth in the
directions tangential to $\partial D_0$ for all $t \in \R.$  More
precisely, the restriction of $X(\cdot,t)$ to $\partial D_0$ is in
$C^{1+\gamma}(\partial D_0, \R^d)$ for all times $t \in \R.$ In
particular, $\partial D_t$ is a domain of class $C^{1+\gamma}$ for
$t \in \R.$ The local existence Theorem \ref{theo2.1} shows that
$X(\cdot,t)$  is in $C^{1+\gamma}(\partial D_0, \R^d)$ for $t
\in(-T,T),$ where $T$ is the small time given by the Picard Theorem.
Assume that $T$ is maximal with the property that the solution
$X(\cdot,t)$  is defined in $(-T,T).$ Our goal is to prove the a
priori estimates on $(-T,T)$ which will let us to conclude that
indeed $T=\infty.$

 It is enough
to prove that $\partial D_t$ is a domain of class $C^{1+\gamma}$ for
all $t \in \R.$ In fact, if this is true, then we have
\begin{equation}\label{lotiX}
\|X\|_{1+\gamma, \partial D_0} < \infty, \quad t \in \R.
\end{equation}
 Otherwise, let $T$ be a maximal time so that \eqref{lotiX} holds for $t \in
 (-T,T).$  Taking  $D_T$ or $D_{-T}$ as initial domain
 at time $T$ or $-T$ (not at time $0$ !) in Theorem \ref{theo2.1}  we contradict the maximality of $T.$

 To show that  $\partial D_t$ is a domain of class $C^{1+\gamma}$ for
all $t \in \R$ take a defining function $\Phi_0$ for $D_0.$ Then
$\Phi_0$ is in $C^{1+\gamma}(\R^d),$
 $D_0 = \{\Phi_0 <0\},$ $\partial
D_0 = \{\Phi_0 = 0\}$ and $\nabla \Phi_0(x) \neq 0$, $x \in \partial
D_0 .$  Consider the equation
\begin{equation}\label{thomas}
\frac{\partial}{\partial t} \Phi(x,t)+ \nabla \Phi(x,t){\cdot} v(x,t)  =
-\Phi(x,t) \chi_{D_t}(x), \quad x \in \R^d, \quad t \in \R,
\end{equation}
with initial condition $\Phi(x,0)= \Phi_0(x).$  Then
$$\Phi(x,t) = \Phi_0(X^{-1}(x,t)), \quad x \in \R^d
\setminus D_t,$$
 and
 $$\Phi(x,t) =
e^{-t}\,\Phi_0(X^{-1}(x,t)),  \quad x \in D_t,$$
 where, for a fixed time $t,$  $X^{-1}(x,t)$
is the inverse of the mapping $x \rightarrow X(x,t).$  Notice that
$\Phi(x,t)$ is continuously differentiable in the open sets $D_t$
and $\R^d \setminus \overline{D_t},$ but that $\nabla \Phi(x,t)$
may, a priori, have a jump at the boundary of $D_t,$ just as $\nabla
X^{-1}(x,t)$ or  $\nabla X(x,t).$ We will show in the next section
that $\Phi(x,t)$ is of class $C^{1+\gamma}(\R^d)$ and thus a
defining function for $D_t.$ We use this fact freely in this
section.

The a priori estimates we need are collected in the following
statement.

\begin{theorem}\label{apriori}
Let $\Phi(\cdot,t)$ the defining function for $D_t$ determined by
\eqref{thomas}. Then
\begin{equation}\label{prioriuniform}
\|\nabla \Phi(\cdot,t)\|_\infty \le \|\nabla \Phi(\cdot,0)\|_\infty
\,\exp \int_0^t (1+\|\nabla v(\cdot,s)\|_\infty)\,ds,
\end{equation}

\begin{equation}\label{prioriinf}
|\nabla \Phi(\cdot,t)|_{\operatorname{inf}} \ge  |\nabla
\Phi(\cdot,0)|_{\operatorname{inf}} \,\exp (-\int_0^t (1+\|\nabla
v(\cdot,s)\|_\infty)\,ds)
\end{equation}
and
\begin{equation}\label{priorigamma}
\|\nabla \Phi(\cdot,t)\|_{\gamma} \le \|\nabla
\Phi(\cdot,0)\|_{\gamma} \,\exp (C \int_0^t (1+\|\nabla
v(\cdot,s)\|_\infty)\,ds).
\end{equation}
\end{theorem}

\begin{proof}
Taking derivatives in \eqref{thomas} we obtain that the material
derivative of $\nabla \Phi$ is
\begin{equation}\label{matder}
\frac{D}{Dt}(\nabla \Phi) = - \nabla v (\nabla \Phi)- \chi_{D_t}
\,\nabla \Phi .
\end{equation}
We have used here that, since $\Phi(x,t)$ vanishes on $\partial
D_t,$
\begin{equation*}
\Phi(x,t)\, \nabla \chi_{D_t}(x)= \Phi(x,t)\, \vec{n}(x)\,
d\sigma(x)=0.
\end{equation*}
By \eqref{matder}
\begin{equation*}
|\nabla \Phi(x,t)| \le |\nabla \Phi(x,0)| + \int_o^t \left(
1+\|\nabla v (\cdot,s)\|_\infty \right) \|\nabla
\Phi(\cdot,s)\|_\infty \,ds
\end{equation*}
and \eqref{prioriuniform} follows from Gr\"onwall.

For \eqref{prioriinf} take $x \in \partial D_t.$ Then
\begin{equation*}\label{loggrad}
\begin{split}
 \frac{D}{D t} \log |\nabla \Phi(x,t)| &= \frac{1}{|\nabla
 \Phi(x,t)|^2}\, \nabla \Phi(x,t)\cdot \frac{D}{D t} (\nabla
 \Phi(x,t)) \\ & \ge - \frac{1}{|\nabla
 \Phi(x,t)|} \left|\frac{D}{D t} (\nabla
 \Phi(x,t))\right| \\ &  \ge -(1+\|v(\cdot,t)\|_\infty)
\end{split}
\end{equation*}
and so
\begin{equation*}
|\nabla \Phi(x,t)| \ge |\nabla \Phi(x,0)| \, \exp\left(-\int_0^t
(1+\|\nabla v (\cdot,s)\|_\infty) \, ds \right),
\end{equation*}
which yields \eqref{prioriinf} at once.

For \eqref{priorigamma} we need two lemmas. The first one is an elementary remark.

\begin{lemma}\label{gradient}
If $D$ is a bounded $C^{1+\gamma}$ domain and $\Phi$ is a defining
function for $D$, then
\begin{equation}\label{lemgrad}
\chi_{D} (x) \Phi(x) = \nabla N * (\chi_D \nabla \Phi)(x), \quad x
\in \R^d.
\end{equation}
\end{lemma}

\begin{proof}
On one hand, the functions in either side of \eqref{lemgrad} are continuous functions. On the other hand,
we have the distributional identities

\begin{equation*}
\nabla N * (\chi_D \nabla \Phi) = \nabla N * \nabla (\chi_D \Phi) = 
\Delta N * \chi_D \Phi = \chi_D \Phi.
\end{equation*}
\end{proof}

Denote by $HN$ the distributional Hessian matrix of $N.$ If $i \neq
j$ the entry in $HN$ corresponding to the $i-$th row and $j-$th
column is the distribution
\begin{equation}\label{dijN}
\operatorname{p.v.}\, \partial_{ij}^2 N(x) =  \operatorname{p.v.} -
\frac{d}{\omega_{d-1}} \frac{x_i x_j}{|x|^{d+2}},
\end{equation}
while the diagonal term corresponding to the indexes $i=j$ is the
distribution
\begin{equation}\label{diiN}
\operatorname{p.v.} \partial_{ii}^2 N(x) + \frac{1}{d} \delta_0 =
\operatorname{p.v.} \frac{1}{\omega_{d-1}} \frac{|x|^2-d
x_i^2}{|x|^{d+2}}  + \frac{1}{d} \delta_0,
\end{equation}
where $\delta_0$ stands for the Dirac delta at the origin. In
\eqref{dijN} and \eqref{diiN} the second order partial derivatives
of $N$ are taken pointwise for $x \neq 0.$ Hence
\begin{equation}\label{HN}
HN =  \operatorname{p.v.} \nabla^2 N + \frac{1}{d} I_0,
\end{equation}
where $I_0$ stands for  the diagonal matrix with $\delta_0$ in the
diagonal. In the next lemma we establish a commutator formula which
is crucial in what follows. In the statement below $\nabla^2 N(x),
\; x \neq 0,$ stands for the $d \times d$ square matrix with entries
the pointwise partial derivatives $\partial_{ij} N(x).$  The
integrand in the right hand side is absolutely integrable because
the vector $\nabla \Phi$ satisfies a H\"{o}lder condition of order
$\gamma.$

\begin{lemma}\label{commutatorformula}
\begin{equation}\label{commfor}
 \frac{D}{Dt}(\nabla \Phi(x,t)) = \int_{D_t} \nabla^2 N(x-y) \left(\nabla \Phi(x)- \nabla \Phi(y) \right) \, dy , \quad x \in \R^d.
\end{equation}
\end{lemma}

\begin{proof}
Since $v = - \nabla N * \chi_{D_t}$, taking gradient we get
\begin{equation}\label{gradv}
\nabla v = - HN * \chi_{D_t} = -\operatorname{p.v.} \nabla^2 N *
\chi_{D_t} - \frac{1}{d} \chi_{D_t} I,
\end{equation}
with $I$ the identity matrix.  Identities  \eqref{matder} and
\eqref{gradv} yield
\begin{equation}\label{matder1}
\frac{D}{Dt}(\nabla \Phi) = \left(\operatorname{p.v.} \nabla^2 N *
\chi_{D_t}\right)(\nabla \Phi) + \frac{1}{d} \chi_{D_t} \nabla \Phi
- \chi_{D_t} \nabla \Phi.
\end{equation}
Taking gradient in \eqref{lemgrad}
\begin{equation}\label{gradchiD}
\chi_{D_t} \nabla \Phi = HN * (\chi_{D_t} \nabla \Phi) =
\operatorname{p.v.} \nabla^2 N * (\chi_{D_t} \nabla \Phi) +
\frac{1}{d} \chi_{D_t} \nabla \Phi.
\end{equation}
Combining \eqref{matder1} and \eqref{gradchiD} completes the proof
of the lemma.
\end{proof}

Lemma \ref{commutatorformula} yields the a priori estimate
\eqref{priorigamma} exactly as in \cite{BC} or \cite{MB}. Theorem
\ref{apriori} is then proved.
\end{proof}
Inserting in \eqref{eq7.6} (the logarithmic estimate of $\|\nabla
v(\cdot,t)\|_\infty$ in terms of $q(D_t)$) the a priori estimates
\eqref{prioriinf} and \eqref{priorigamma} we get

\begin{equation*}\label{gradvuniform}
||\nabla v (x,t)||_{\infty} \le C + C \int_0^t ||\nabla v
(\cdot,\tau)||_{\infty} \, d\tau,
\end{equation*}
which yields by Gr\"onwall's inequality
\begin{equation}\label{expgradv}
||\nabla v (x,t)||_{\infty} \le C \,e^{C t}, \quad -T< t< T.
\end{equation}
 Upon introducing this in  \eqref{prioriuniform} and \eqref{prioriinf} we
conclude that we have the double exponential estimate
\begin{equation}\label{expqD}
q(D_t) =
\frac{\|\nabla\Phi(\cdot,t)\|_\gamma}{|\nabla\Phi(\cdot,t)|_{inf}}\le
C \exp(e^{Ct}), \quad -T< t< T.
\end{equation}
For fixed $t,$ the flux $\alpha \rightarrow X(\alpha,t)$ is a
bilipschitz homeomorphism. In particular, we have
\cite[(4.47),p.149]{MB}
\begin{equation}\label{fluxbi}
||\nabla X(\alpha,t)||_{\infty} \leq \exp\left(\int_0^t ||\nabla
v(\cdot,s)||_{\infty}\,ds \right),
\end{equation}
and so
\begin{equation*}\label{length}
\begin{split}
\sigma(\partial D_t) & = \int_{\partial D_0}
\Big|\bigwedge^{d-1}_{j=1} DX (\beta, t)
(T_{j}(\beta))\Big|\,d\sigma(\beta) \\& \le \int_{\partial D_0}
(d-1)^{\frac{1}{2}}\,\|DX(\cdot,t)\|_{\infty}^{d-1}\,d\sigma(\beta)
\\& \le
 (d-1)^{\frac{1}{2}}\, \exp\left((d-1) \int_0^t ||\nabla
v(\cdot,s)||_{\infty}\,ds \right)\, \sigma(\partial D_0).
\end{split}
\end{equation*}
Hence, by \eqref{expgradv},
\begin{equation}\label{explength}
\sigma(\partial D_t) \le  (d-1)^{\frac{1}{2}}\,\sigma(\partial
D_0)\, \exp(C\, e^{Ct}) , \quad -T< t< T.
\end{equation}
A similar estimate of the diameter of $D_t$ follows from
\eqref{expgradv} and \eqref{fluxbi}:
\begin{equation}\label{expdiam}
\operatorname{diam}(D_t) \le \operatorname{diam}(D_0)\, \exp(C\,
e^{Ct}), \quad -T< t< T.
\end{equation}
 We can combine the estimates
\eqref{expqD}, \eqref{explength} and \eqref{expdiam}  with Theorem
\ref{theo2.1} to show that $D_t$
 is a  $C^{1+\gamma}$ domain for all $t \in \R$ and thus complete the proof of our main result. For that, assume that $T <
 \infty$ is the maximal time for which $D_t$ is a  $C^{1+\gamma}$ domain for all $t \in (-T,T)$.
 Since the size of the interval in which the local solution of
 Theorem  \ref{theo2.1} exists depends only on the quantities  $q(D_0), \sigma(\partial D_0)$ and
$\operatorname{diam}(D_0)$   and these quantities are uniformly
bounded for $D_t$ with $t \in (-T,T),$ we can apply Theorem
\ref{theo2.1} with initial condition $D_{\tau_0}$ at time $\tau_0$
(not at time $0$!) with $\tau_0$ sufficiently close to $T$ so that
the new interval of existence goes beyond $T.$ This contradicts the
maximality of $T.$

Notice that we have concluded the proof without proving any a priori
estimate for $\|D X(\cdot,t)\|_{\gamma, \partial D_0}.$ This is the
reason why we took particular care in estimating the local time of
existence in terms of quantities related to the smoothness of the
initial domain $D_0.$ A final remark on inequality \eqref{expdiam} is in order.
One can easily prove the much better estimate $\operatorname{diam}(D_t) \le C_d\, \operatorname{diam}(D_0)$, 
with $C_d$ a dimensional constant, using the straightforward potential theoretic estimate
$|v(x,t)| \le C_d\, |D_t|^{1/d}, \; x \in \mathbb{R}^d, \; t \in \mathbb{R}$, and the fact that
$ |D_t| = e^{-t}|D_0|.$

\section{The gradient of $\Phi$ has no jump}\label{sec8}

In this section we prove that the function $\Phi(x,t)$  defined by
\eqref{thomas} is of class~$C^{1+\gamma}(\mathbb{R}^{d})$.

Let $X(\alpha,t)$ be the Yudovich flow \eqref{flux0} and
\eqref{velocity0}. The initial domain $D_{0}$ is bounded and has
boundary of class~$C^{1+\gamma}$, $0<\gamma <1$. By
Theorem~\ref{theo2.1} we know that for some $T>0$
\begin{equation}\label{fluxX}
X(\alpha,t)\in C^{1} \bigl((-T,T), C^{1+\gamma} (\partial
D_{0},\mathbb{R}^{d})\bigr).
\end{equation}
Fix a time $t,\; $$0<t<T$. By \eqref{eq7.6} $\|\nabla
v(\cdot,t)\|_{\infty}$ is finite and controlled by the constant
$q(D_t)$ describing the $C^{1+\gamma}$ character of $\partial
D_{t}$. In view of \eqref{fluxX}  this constants are uniformly
bounded for $|\tau|\le t$. Hence, by \eqref{eq7.6},
\begin{equation}\label{gradvinf}
\|\nabla v(\cdot,\tau)\|_{\infty} \le q,\quad |\tau|\le t,
\end{equation}
for a positive constant $q$ depending only on $t.$ We know from
Lemma \ref{eq7.1} that the entries of the matrix~$\nabla v(\cdot,t)$
are given by singular integrals with even kernels plus a scalar
multiple of~$\newchi_{D_t}$. By the main lemma of~\cite{MOV} $\nabla
v(\cdot, t)$ satisfies a H\"older condition of order $\gamma$ on
each of the open sets~$D_{t}$ and $\mathbb{R}^{d}\setminus
\overline{D_{t}}$, in spite of having a jump at~$\partial D_{t}$.
Again the constants of this H\"older conditions are controlled by
the $C^{1+\gamma}$ character of~$\partial D_{t}$. Therefore, for
some other constant $q$ depending only on $t,$
\begin{equation}\label{gradvgam}
\|\nabla v(x_1,\tau) - \nabla v (x_{2},\tau)\| \le q \;
|x_{1}-x_{2}|^{\gamma}, \quad |\tau|\le t,
\end{equation}
provided $x_{1},x_{2}\in D_{\tau}$ or $x_{1},x_{2}\in
\mathbb{R}^{d}\setminus \overline{D_{\tau}}$. The estimates
\eqref{gradvinf} and \eqref{gradvgam} imply that $\nabla X(\cdot,t)$
extends continuously to~$\partial D_{0}$ from either side. In the
same vein, $\nabla X^{-1}(\cdot,t)$ extends continuously
to~$\partial D_{t}$  from either side. This follows from standard
estimates and Gr\"onwall's inequality, as we recall below for the
sake of the reader. One starts by
$$
X(\alpha,t)=\alpha +\int^{t}_{0} v(X(\alpha,\tau),\tau)\,d\tau_.
$$
Using \eqref{gradvinf} we obtain
$$
|X(\alpha,t)-X(\beta,t)| \le |\alpha-\beta| +q \int^{t}_{0} |X
(\alpha,\tau)-X(\beta,\tau)|\,d\tau_,
$$
which yields by Gr\"onwall
$$
|X(\alpha,t)- X(\beta,t)|\le |\alpha-\beta|\, e^{qt}.
$$
Since
$$
\nabla X (\alpha,t)=I+\int^{t}_{0} \nabla v
(X(\alpha,\tau),\tau)\circ \nabla X(\alpha,\tau)\,d\tau,
$$
we get
$$
\|\nabla X(\cdot,t)\|_{\infty} \le e^{qt}
$$
and
$$
\|\nabla X(\alpha,t)-\nabla X(\beta,t)\| \le q\,
|\alpha-\beta|^{\gamma} \int^{t}_{0}
e^{q(\gamma+1)\tau}\,d\tau+\int^{t}_{0} q \|\nabla
X(\alpha,\tau)-\nabla X (\beta,\tau)\|\,d\tau,
$$
provided $\alpha,\beta \in D_{0}$ or $\alpha,\beta \in
\mathbb{R}^{d}\setminus \overline{D_{0}}$. Thus
$$
\|\nabla X(\alpha,t)-\nabla X(\beta,t)\| \le \frac{1}{\gamma}
e^{q(\gamma+1)t}|\alpha-\beta|^{\gamma},
$$
for $\alpha,\beta \in D_{0}$ or $\alpha,\beta \in
\mathbb{R}^{d}\setminus \overline{D_{0}}$.

Since
$$
\frac{d}{d\tau} X^{-1} (x,\tau) = -v (X^{-1}(x,\tau),t-\tau)
$$
the same argument applied to~$X^{-1}$ yields
$$
\|\nabla X^{-1}(x,t)-\nabla X^{-1}(y,t)\| \le \frac{1}{\gamma}
e^{q(\gamma+1)t} |x-y|^{\gamma},
$$
for $x,y\in D_{t}$ or $x,y\in \mathbb{R}^{d}\setminus
\overline{D_{t}}$.

For $x\in\partial D_{t}$ set
$$
M=M(x,t) =\lim_{D_{t}\ni y\to x} \nabla X^{-1}(y,t)
$$
and
$$
N=N(x,t) =\lim_{\mathbb{R}^{d}\setminus \overline{D_{t}}\ni y\to x}
\nabla X^{-1}(y,t).
$$
Then $M(x,t)$ and $N(x,t)$ are linear mappings from $\R^d$ into
itself that depend continuously on $x \in \partial D_t.$ Let
$\operatorname{Tan} (\partial D_{t},x)$ stand for the tangent space
to~$\partial D_{t}$ at the point~$x\in
\partial D_{t}$.

\begin{lemma}\label{Mboundary}
\hspace{-1pt} For $x \in \partial D_t$ the linear mappings $M\!$ and
$N\!$ coincide on $\operatorname{Tan} (\partial D_{t},x)$ with the
differential at~$x\!$ of~$X^{-1}(x,t),$ viewed as a differentiable
mapping from~$\partial D_{t}$ onto $\partial D_{0}$. In particular,
$M$ and $N$ map $\operatorname{Tan}(\partial D_{t},x)$ into
$\operatorname{Tan} (\partial D_{0},X^{-1}(x,t))$.
\end{lemma}

\begin{proof}
Let $T$ be a tangent vector to~$\partial D_{t}$ at~$x$, $a>0$ and
$z\colon (-a,a) \to \partial D_{t}$ a curve of class~$C^{1}$ such
that $z(0)=x$ and $z'(0)=T$. Let $\vec{n}$ be the exterior unit
normal vector to~$\partial D_{t}$ at~$x$. Let $0<\varepsilon <a$ so
small that $z(\theta)-\eta \vec{n}\in D_{t}$ provided
$|\theta|<\varepsilon$ and $0<\eta <\varepsilon$. Then
$$
X^{-1}(z(\theta)-\eta \vec{n},t) \xrightarrow{\eta\to 0}
X^{-1}(z(\theta),t)\text{ uniformly in } \theta \in
(-\varepsilon,\varepsilon).
$$
Hence
$$
\frac{d}{d\theta} X^{-1} (z(\theta)-\eta \vec{n},t)
\xrightarrow{\eta \to 0} \frac{d}{d\theta} X^{-1} (z(\theta),t)
$$
as distributions on $(-\varepsilon,\varepsilon)$.

On the other hand, for $\theta\in (-\varepsilon,\varepsilon)$ we
have
$$
\frac{d}{d\theta} X^{-1} (z(\theta)-\eta \vec{n},t)=\nabla
X^{-1}(z(\theta)-\eta\vec{n},t)\left(\frac{dz(\theta)}{d\theta}\right)
 \xrightarrow{\eta \to 0} M (z(\theta),t)\left(  \frac{dz(\theta)}{d\theta} \right)
$$
pointwise and boundedly. Thus
$$
\frac{d}{d\theta} X^{-1} (z(\theta),t)=M
(z(\theta),t)\left(\frac{dz(\theta)}{d\theta}\right), \quad \theta
\in (-\epsilon, \epsilon),
$$
which, for $\theta=0$, gives
$$
 DX^{-1}(T) = M(T),
$$
where $DX^{-1}$ is the differential at~$x$ of $X^{-1}$ as a smooth
map from $\partial D_{t}$ into $\partial D_{0}$. The argument for
$N$ is similar. The proof is complete.
\end{proof}

Let $\Phi_{0}$ be a defining function for $D_{0}$. Then $\Phi_{0}\in
C^{1+\gamma}(\mathbb{R}^{d})$, $D_{0}= \{ x\in \mathbb{R}^{d}:
\Phi_{0}(x)<0\},$ $\partial D_{0} = \{ x\in \mathbb{R}^{d}:
\Phi_{0}(x)=0\}$ and $\nabla \Phi_{0}(x)\ne 0$, $x\in
\partial D_{0}$. Set
$$
\varphi(x,t)=\Phi_{0}(X^{-1}(x,t)),\quad x\in\mathbb{R}^{d},
$$
so that $D_{t}= \{ x\in \mathbb{R}^{d}: \varphi(x,t)<0\}$, $\partial
D_{t}= \{ x\in \mathbb{R}^{d}: \varphi(x,t)=0\}$ and $ \nabla
\varphi (y,t)=\nabla \Phi_{0}\circ \nabla X^{-1}(y,t), $  \; for $y
\notin \partial D_t.$  Our next task is computing the jump of
$\nabla \varphi$ at $\partial D_t.$

Let $\vec{n}$ and $\overrightarrow{n_{0}}$ be the exterior unit
normal vectors to $\partial D_{t}$ and $\partial D_{0}$
respectively, at the points~$x \in \partial D_t$ and
$X^{-1}(x,t)=\alpha \in \partial D_0$. Given any
vector~$\vec{u}\in\mathbb{R}^{d}$, we have
$$
\lim_{D_{t}\ni y\to x} \langle \nabla \varphi (y,t),\vec{u}\rangle =
\langle \nabla \Phi_{0}(\alpha), \lim_{D_{t}\ni y\to x} \nabla
X^{-1}(y,t)(\vec{u}) \rangle = \langle \nabla \Phi_{0}(\alpha),
M(\vec{u})\rangle.
$$
Set $\vec{u} = \lambda \vec{n} + T$, $\lambda\in\mathbb{R}$ and
$T\in \operatorname{Tan} (\partial D_{t},x)$, and
$$
M(\vec{n})=A \overrightarrow{n_{0}}+T_{0},\quad A\in\mathbb{R},
\quad T_{0}\in\operatorname{Tan} (\partial D_{0},\alpha).
$$
Then
$$
M(\vec{u})= \lambda M(\vec{n})+ M(T) =\lambda
A\overrightarrow{n_{0}} +\lambda T_{0}+M(T)
$$
and, since $\lambda T_{0}+ M(T)\in\operatorname{Tan} (\partial
D_{0},\alpha)$,
$$
\langle \nabla \Phi_{0}(\alpha), M(\vec{u})\rangle =\lambda A\langle
\nabla \Phi_{0}(\alpha),\overrightarrow{n_{0}}\rangle =\lambda A
|\nabla \Phi_{0}(\alpha)| = A |\nabla \Phi_{0}(\alpha)| \langle
\vec{n}, \vec{u}\rangle.
$$
Therefore
$$
\lim_{D_{t}\ni y\to x} \nabla \varphi (y,t) =A|\nabla
\Phi_{0}(\alpha)| \vec{n}.
$$
Take an orthonormal basis $\{\tau_{1},\dotsc,\tau_{d-1}\}$ of
$\operatorname{Tan}(\partial D_{t},x)$ and an orthonormal basis $\{
\tau_{1}^{0},\dotsc, \tau^{0}_{d-1}\}$ of $\operatorname{Tan}
(\partial D_{0},\alpha)$ so that
$$
\det (\vec{n}, \tau_{1},\dotsc, \tau_{d-1})=\det
(\overrightarrow{n_{0}}, \tau_{1}^{0},\dotsc,\tau^{0}_{d-1})=1.
$$
Call $D$ the differential of $X^{-1}(x,t)$ at~$x$ as a smooth
mapping from~$\partial D_{t}$ into $\partial D_{0}$. Then the matrix
of~$M$ in the above basis is
$$
M= \left(
\begin{array}{cccc}
A &0 &\hdots &0\\
A_{1}\\
\vdots &\multicolumn{3}{c}{\text{Matrix of $D$}}\\
A_{d-1}
\end{array}
\right).
$$
Taking determinants
$$
\det M=A\det D.
$$
Now $\det M$ is the limit  of $\det \nabla X^{-1}(y,t)$ as $y \in
D_t$ tends to $x,$  which turns out to be $e^{t}.$  This is so
because $\det \nabla X(\alpha,t) = e^{-t}$ for $\alpha \in D_0,$
which in turn is due to the the well-known fact that (\cite[p.
5]{MB})
$$
\frac{d}{dt} \det \nabla X(\alpha,t) = \operatorname{div}
v(X(\alpha,t),t)\,\det \nabla X(\alpha,t)= - \det \nabla
X(\alpha,t).
$$
 Therefore
$$
\lim_{D_{t}\ni y\to x} \nabla \varphi (y,t) = e^{t} \frac{|\nabla
\Phi_{0}(\alpha)|} {\det D} \vec{n}.
$$
Arguing similarly for the exterior side, where $\det \nabla 
X^{-1}(y,t)$ is $1$, we conclude that
$$
\lim_{\mathbb{R}^{d}\setminus \overline{D_{t}}\ni y\to x} \nabla
\varphi(y,t)=\frac{|\nabla \Phi_{0}(\alpha)|}{\det D} \vec{n}.
$$
Then, clearly, the function
$$
\Phi (y,t)=e^{-t}\varphi (y,t) \newchi_{D_{t}}(y)
+\varphi(y,t)\newchi_{\mathbb{R}^{d}\setminus \overline{D_{t}}}(y)
$$
has no gradient jump at $\partial D_{t}$ and so $\Phi(y,t)$ is a
defining function for $D_{t}$ of class~$C^{1+\gamma}$. The material
derivative of $\Phi(y,t)$ is
$$
\frac{D\Phi (y,t)}{D t}=-\Phi (y,t) \newchi_{D_{t}}(y)
$$
and $\Phi (y,0)=\Phi_{0}(y)$. Hence the function determined by
\eqref{thomas} is of class~$C^{1+\gamma}$, as desired.

\begin{acknowledgements}
A.Bertozzi was partially supported by NSF grants CMMI-1435709 and DMS-0907931 ; J. Garnett by NSF grant
DMS 1217239;  T. Laurent by NSF DMS 1414396; J. Verdera by  MTM2013-44699-P
(MINECO) and 2014SGR75 (Generalitat de Catalunya). J.Verdera is grateful to the Department of Mathematics of
UCLA for their hospitality in October 2012.
 \end{acknowledgements}

\bibliographystyle{plain}

\begin{tabular}{l}
Andrea L. Bertozzi\\
Department of Mathematics\\
University of California at Los Angeles\\
Los Angeles, CA 90095\\
\\
{\it E-mail:} {\tt bertozzi@math.ucla.edu}\\ \\
John B. Garnett\\
Department of Mathematics\\
University of California at Los Angeles\\
Los Angeles, CA 90095 \\
\\
{\it E-mail:} {\tt jbg@math.ucla.edu}\\ \\
Thomas B. Laurent\\
Department of Mathematics\\
Loyola Marymount University\\
Los Angeles, CA 90045 \\
\\
{\it E-mail:} {\tt tlaurent@lmu.edu}\\ \\
Joan Verdera\\
Departament de Matem\`{a}tiques\\
Universitat Aut\`{o}noma de Barcelona\\
08193 Bellaterra, Barcelona, Catalonia\\
\\
{\it E-mail:} {\tt jvm@mat.uab.cat}\\ \\
\end{tabular}
\end{document}